\documentclass[11pt]{amsart}
\usepackage{amsfonts}
\usepackage{amsmath}
\usepackage{color}
\usepackage[mathscr]{eucal}
\usepackage{amscd, latexsym, graphicx, mathrsfs, enumerate}
\usepackage{amssymb}

\makeindex

\newtheorem{thm}{{{Theorem}}}[section]
\newtheorem{prop}[thm]{{Proposition}}
\newtheorem{lem}[thm]{{Lemma}}
\newtheorem{cor}[thm]{{Corollary}}
\newtheorem{Conj}{{{Conjecture}}}[section]

\numberwithin{equation}{section}
\newtheorem{Def}[equation]{Definition}

\def\Z{\mathbb{Z}}
\def\Q{\mathbb{Q}}
\def\R{\mathbb{R}}
\def\C{\mathbb{C}}
\def\A{\mathbb{A}}

\def\diag{{\mathop{\mathrm{diag}}}}

\def\vol{{\mathop{\mathrm{vol}}}}

\def\ve{\varepsilon}

\def\k{\underline{k}}

\def\inf{\infty}
\def\fin{{\mathrm{fin}}}

\def\bs{{\backslash}}
\def\ds{\displaystyle}
\def\bQ{\overline{\Q}}

\def\lra{{\longrightarrow}}
\def\G{{\Gamma}}
\def\br{\overline{\rho}}

\numberwithin{equation}{section}

\title[Hecke fields and $n$-level density]
{Equidistribution theorems for holomorphic Siegel modular forms for $GSp_4$; Hecke fields and $n$-level density}
\author{Henry H. Kim, Satoshi Wakatsuki and Takuya Yamauchi}
\thanks{The first author is partially supported by NSERC. The second author is partially supported by JSPS Grant-in-Aid for Scientific Research (No. 26800006, 25247001, 15K04795). The third author is partially supported by JSPS Grant-in-Aid for Scientific Research (C) No.15K04787.}
\address{Henry H. Kim \\
Department of mathematics \\
 University of Toronto \\
Toronto, Ontario M5S 2E4, CANADA \\
and Korea Institute for Advanced Study, Seoul, KOREA}
\email{henrykim@math.toronto.edu}

\address{Satoshi Wakatsuki \\
Faculty of Mathematics and Physics, Institute of Science and Engineering\\
Kanazawa University\\
Kakumamachi, Kanazawa, Ishikawa, 920-1192, JAPAN}
\email{wakatsuk@staff.kanazawa-u.ac.jp}

\address{Takuya Yamauchi \\
Mathematical Inst. Tohoku Univ.\\
 6-3,Aoba, Aramaki, Aoba-Ku, Sendai 980-8578, JAPAN}
\email{yamauchi@math.tohoku.ac.jp}

\begin{document}
\begin{abstract}
This paper is a continuation of \cite{KWY}. We supplement four results on a family of holomorphic Siegel cusp forms for $GSp_4/\Q$.
First, we improve the result on Hecke fields. Namely, we prove that the degree of Hecke fields is unbounded on the subspace of genuine forms which do not come from functorial lift of smaller subgroups of $GSp_4$  under a conjecture in local-global compatibility
and Arthur's classification for $GSp_4$.
Second, we prove simultaneous vertical Sato-Tate theorem. Namely, we prove simultaneous equidistribution of Hecke eigenvalues at finitely many primes.
Third, we compute the $n$-level density of degree 4 spinor $L$-functions, and thus we can distinguish the symmetry type depending on the root numbers. This is conditional on certain conjecture on root numbers.
Fourth, we consider equidistribution of paramodular forms. In this case, we can prove a result on root numbers.
Main tools are the equidistribution theorem in our previous work and Shin-Templier's work \cite{ST1}.
\end{abstract}
\maketitle


\section{Introduction}
This paper is a continuation of \cite{KWY}. We use the same notations throughout this paper. We study Hecke fields, simultaneous vertical Sato-Tate theorem, and the $n$-level density on a family of
holomorphic Siegel cusp forms for $GSp_4/\Q$.

First, Hecke fields. Let $S^1_k(\Gamma_0(N))$ be the space of elliptic cusp forms of weight $k\ge 2$ with respect to a congruent subgroup $\Gamma_0(N)$.
The Hecke operators $\{T_p\}_{p\nmid N}$ acting on the space and it has a basis consisting of simultaneous eigenforms in
Hecke operators which are called  normalized Hecke eigenforms. Let $f$ be such an eigenform and $a_p(f)$ be
the Hecke eigenvalue of $T_p$ for $p\nmid N$, i.e. $T_pf=a_p(f)f$. The field $\Q_f:=\Q(a_p(f)\ |\ p\nmid N)$
generated by such eigenvalues over $\Q$ is called the Hecke field of $f$. Since $S^1_k(\Gamma_0(N))$ has an integral
structure preserved by Hecke operators, the eigenvalues are algebraic numbers and it turns out that $\Q_f$ is a
finite extension over $\Q$.
The Hecke field of $f$ reflects various arithmetic properties of $f$ and has been studied by many people \cite{Shimura}, \cite{Serre}.
For example $\Q$-simple factors of the Jacobian $J_0(N)$ of the modular curve $X_0(N)$ can be described in terms of the
degree of the Hecke fields and one can ask the maximal dimension of $\Q$-simple factors for $J_0(N)$
(see \cite{Mazur}, \cite{Serre}, \cite{Royer},\cite{Yam} \cite{MS}).

Let $S_{\underline{k}}(\G(N),\chi)$ be the space of classical holomorphic Siegel cusp forms of degree 2 with the level $\G(N)$,
a central character $\chi:(\Z/N\Z)^\times\lra \C^\times$, and weight $\underline{k}=(k_1,k_2), k_1\ge k_2\ge 3$
(cf.  Section 2 of \cite{KWY}).
For a prime $p\nmid N$, let $T(p^n)$ be the Hecke operator with the similitude $p^n$.
Any eigenform with respect to $T(p^n)$ for any non-negative integer $n$ and any prime $p\nmid N$ is called
a Hecke eigen cusp form. We denote by $HE_{\underline{k}}(\G(N),\chi)$ the set of all such eigenforms in
$S_{\underline{k}}(\G(N),\chi)$.

Let $F$ be a Hecke eigen cusp form and $\lambda_F(p^n)$ be the Hecke eigenvalue of $F$ for $T(p^n)$, i.e.
$T(p^n)F=\lambda(p^n)F$. It is known that $\lambda_F(p^n)$ is an algebraic integer (cf. Lemma 2.1 of \cite{Taylor-thesis}).
We consider the Hecke field $\Q_F:=\Q(\lambda_F(p^n),\ \chi(p),\ \chi_2(p)\ |\ p\nmid N)$ which turns out to be a finite extension over $\Q$ (see (2.12) of \cite{KWY} for $\chi_2$).
We call $F$ a genuine form if it never comes from any functorial lift from a smaller subgroup of $GSp_4$, hence,
it is neither a CAP form, an endoscopic lift, a base change lift, an Asai lift, nor a symmetric cubic lift (see Section \ref{smf} for the details).

\begin{thm}\label{hecke1} Assume  Conjecture 1 of \cite{Jorza} and Arthur's classification for $GSp_4$
Fix a weight $\underline{k}=(k_1,k_2)$ with $k_1\ge k_2\ge 3$ and a prime $p$. Then
$$\limsup_{N\to \infty,  p\nmid N,\ (N,11!)=1\atop 
{\rm ord}_\ell(N)\ge 4\ {\rm if}\ \ell|N }\{[\Q_F:\Q]\ |\ F\in
HE_{\underline{k}}(\G(N),\chi):\text{genuine} \}=\infty.$$
\end{thm}

\begin{thm}\label{hecke2} Let the notation and assumptions be as above. Let $p$ be a fixed prime.
Put $d_{\underline{k},N}(\chi)=|HE_{\underline{k}}(\G(N),\chi)|$
Then
$$|\{F\in HE_{\underline{k}}(\G(N),\chi)\ |\ [\Q_F:\Q]\le A \}|=
O\left(\frac{d_{\underline{k},N}(\chi)}{\log d_{\underline{k},N}(\chi)}\right)=O\left(\frac{N^{9+\ve}}{\log N}\right)
$$
as $N$ goes to infinity satisfying either of the following conditions:
\begin{enumerate}
\item ${\rm ord}_\ell(N)\ge 4\ {\rm if}\ \ell|N $ and $(N,p)=1$ ; or
\item ${\rm ord}_p(N)\to \infty$ as $N\to \infty$.
\end{enumerate}
The claim is still true even if we replace $HE_{\underline{k}}(\G(N),\chi)$ with
$\{F\in HE_{\underline{k}}(\G(N),\chi)|\ F:\text{genuine}\}$ but keeping the additional condition $(N,11!)=1$.
Note that the cardinality of that subspace is approximately equal to $d_{\underline{k},\chi}(N)$ when $N, (N,11!)=1$  goes to infinity.
\end{thm}

\begin{thm}\label{hecke3} Keep the condition in Theorem \ref{hecke2}. Put $f(N)=(\log \log N)^\frac{1}{2}$ if $N$ is in the first
case of the two conditions in Theorem \ref{hecke2} and $f(N)=(\log N)^\frac{1}{4}$ otherwise.
Then
$${\rm inf} \{[\Q_F:\Q]\ |\ F\in
HE_{\underline{k}}(\G(N),\chi):\text{genuine} \}\gg f(N)$$ as $N$ goes to infinity with $(N,11!)=1$.
\end{thm}

Second, simultaneous vertical Sato-Tate theorem. Let $S_{\underline{k}}(\Gamma(N),\chi)^{\rm tm}$ be the subspace of $S_{\underline{k}}(\Gamma(N),\chi)$ generated by Hecke eigen forms
$F$ outside $N$ so that $\pi_{F,p}$ is tempered for any $p\nmid N$. Let $HE_{\underline{k}}(\G(N),\chi)^{\rm tm}=S_{\underline{k}}(\G(N),\chi)^{\rm tm}\cap HE_{\underline{k}}(\G(N),\chi)$. For a prime $p\nmid N$, let $a_{F,p}, b_{F,p}\in [-2,2]$ be Hecke eigenvalues as in \cite{KWY}. Then in Section \ref{sato-tate}, we generalize Theorem 1.4 of \cite{KWY} to finitely many primes. Namely, given finitely many distinct primes $p_1,...,p_r$, $((a_{F,p_1}, b_{F,p_1}),...,(a_{F,p_r},b_{F,p_r}))$ is equidistributed with respect to a suitable measure (Theorem \ref{Sato-Tate}).

Third, $n$-level density of degree 4 spinor $L$-functions. In \cite{KWY}, we studied the one-level density of degree 4 spinor $L$-functions of a family of holomorphic Siegel cusp forms for $GSp_4/\Q$, and we showed that its symmetry type is SO(even), SO(odd), or O type, as predicted by \cite{KS1}. However, we could not distinguish the symmetry type among SO(even), SO(odd), and O type since the support of $\hat\phi$ is smaller than $(-1,1)$. In order to distinguish them, we need to compute the $n$-level density.
For the degree 4 spinor $L$-function $L(s,\pi_F, {\rm Spin})$, we denote the non-trivial zeros of $L(s,\pi_F,{\rm Spin})$ by $\frac 12 + \sqrt{-1}\gamma_j$. Let $\phi(x_1,...,x_n)=\prod_{i=1}^n \phi_i(x_i)$ be an even Schwartz class function in each variables whose Fourier transform $\hat\phi(u_1,...,u_n)$
is compactly supported. We define $D^{(n)}(\pi_F, \phi, {\rm Spin})$
as in Section \ref{spinor}.

We first prove the following theorem which may be of independent interest.
Let $L(s,\pi_F,{\rm Spin})=\sum_{m=1}^\infty \tilde\lambda_F(m)m^{-s}$. Let $m=\prod_{p|m} p^{v_p(m)}$.
For simplicity, denote $S_{\underline{k}}(\Gamma(N),1), d_{\underline{k},N}(1)$ by $S_{\underline{k}}(N), d_{\underline{k},N}$, resp.

\begin{thm}\label{sum-formula}
Put $\underline{k}=(k_1,k_2),\ k_1\ge k_2\ge 3$.
\begin{enumerate}
\item (level-aspect) Fix $k_1,k_2$. Then as $N\to\infty$,
$$\frac{1}{d_{\underline{k},N}}\sum_{F\in HE_{\underline{k}}(N)} \tilde\lambda_F(m) =\delta_\square m^{-\frac 12}\prod_{p|m} (1+p^{-2}+\cdots+p^{-v_p(m)}) +O(N^{-2}m^{c}),
$$
where $\delta_\square=\begin{cases} 1, &\text{if $m$ is a square}\\ 0, &\text{otherwise}\end{cases}$.
\item (weight-aspect) Fix $N$. Then as $k_1+k_2\to\infty$,
\begin{eqnarray*}
&& \frac{1}{d_{\underline{k},N}}\sum_{F\in HE_{\underline{k}}(N)} \tilde\lambda_F(m) =\delta_\square m^{-\frac 12}\prod_{p|m} (1+p^{-2}+\cdots+p^{-v_p(m)}) \\
&&\phantom{xxxxxxxxxxxxxx} +O\left(\frac {m^c}{(k_1-1)(k_2-2)}\right)+ O\left(\frac {m^{-d}}{(k_1-k_2+1)(k_1+k_2-3)}\right),
\end{eqnarray*}
for some constants $c,d>0$.
\end{enumerate}
\end{thm}

For one-level density, the root number $\epsilon(\pi_F)$ did not play a role. However, higher level density depends on the root number.

When $N=1$ (i.e., level one case), we have $\epsilon(\pi_F)=(-1)^{k_2}$ (\cite{Sch}). (In this case, $k_1-k_2$ should be even.)
Then we have the following $n$-level density in the weight aspect. Let $HE_{\underline{k}}=HE_{\underline{k}}(1)$ and
$d_{\underline{k}}=d_{\underline{k},1}$.

\begin{thm}\label{level-one}
Let $\phi(x_1,...,x_n)=\phi_1(x_1)\cdots \phi_n(x_n)$, where each $\phi_i$ is an even Schwartz function and $\hat \phi(u_1,...,u_n)=\hat \phi_1(u_1)\cdots \hat\phi_n(u_n)$. Assume the Fourier transform $\hat{\phi_i}$ of $\phi_i$ is supported in $(-\beta_n,\beta_n)$ for $i=1,\cdots,n$. ($\beta_n<1$ can be explicitly determined.)
Then
$$
\frac 1{d_{\underline{k}}} \sum_{F\in HE_{\underline{k}}} D^{(n)}(\pi_F,\phi,{\rm Spin})=
\begin{cases}
\int_{\Bbb R^n} \phi(x)W({\rm SO(even)})(x)\, dx+ O\left(\frac 1{\log c_{\underline{k}}}\right), &\text{if $k_2$ is even}\\
\int_{\Bbb R^n} \phi(x)W({\rm SO(odd)})(x)\, dx+ O\left(\frac 1{\log c_{\underline{k}}}\right), &\text{if $k_2$ is odd}
\end{cases},
$$
where $W({\rm SO(even)})$ and $W({\rm SO(even)})$ are the $n$-level density functions defined in Section \ref{spinor},
and $c_{\underline{k}}=c_{\underline{k},1}$ is the analytic conductor defined in Section \ref{spinor}.
\end{thm}

Let $S^{\pm}_{\underline{k}}(N)$ be the subspace of $S_{\underline{k}}(N)$ with the root number $\epsilon(\pi_F)=\pm 1$.
Let $HE^{\pm}_{\underline{k}}(N)$ be a basis of $S^\pm_{\underline{k}}(N)$ consisting of Hecke eigenforms outside $N$, and
denote $|HE^{\pm}_{\underline{k}}(N)|=d^\pm_{\underline{k},N}$.
When $N$ is large, we expect $d^\pm_{\underline{k},N}=\frac 12 d_{\underline{k},N}+O(N^{9-\epsilon})$, and the analogue of
Theorem \ref{sum-formula} holds when we replace $HE_{\underline{k}}(N)$ by $HE^\pm_{\underline{k}}(N)$.
We assume it as Conjecture \ref{conj}. Then we can prove the $n$-level density result for $HE^\pm_{\underline{k}}(N)$
(Theorem \ref{n-level-spin-even} and Theorem \ref{n-level-spin-odd}).

In Section \ref{Standard}, we study the $n$-level density of the degree 5 standard $L$-functions of holomorphic Siegel cusp forms.
We show that the root number is always one. Hence the symmetry type of the $n$-level density should be Sp.
Under Conjecture \ref{stan} which is an analogue of Conjecture \ref{conj}, we show

\begin{thm}\label{n-level-sp}
Let $\phi(x_1,...,x_n)=\phi_1(x_1)\cdots \phi_n(x_n)$, where each $\phi_i$ is an even Schwartz function and $\hat \phi(u_1,...,u_n)=\hat \phi_1(u_1)\cdots \hat\phi_n(u_n)$. Assume the Fourier transform $\hat{\phi_i}$ of $\phi_i$ is supported in $(-\beta_n,\beta_n)$ for $i=1,\cdots,n$.
Then
$$
\frac 1{d_{\underline{k},N}} \sum_{F \in HE_{\underline{k},N}} D^{(n)}(\pi_F,\phi,{\rm St})=
\int_{\Bbb R^n} \phi(x)W({\rm Sp})(x)\, dx+O\left(\frac {\omega(N)}{\log c_{\underline{k},N}}\right).
$$
\end{thm}

Fourth, we consider paramodular forms. In our previous paper \cite{KWY}, we considered only principal congruence subgroup $\Gamma(N)$. Here we can deal with the paramodular group $K^{\rm para}(N)$. We prove equidistribution results on paramodular forms.
In particular we can show Conjecture \ref{conj} for paramodular forms. Hence $n$-level density for spinor $L$-functions of paramodular forms for weight aspect
(analogues of Theorem \ref{n-level-spin-even} and Theorem \ref{n-level-spin-odd}) hold. In a similar way, we can show simultaneous vertical Sato-Tate theorem for paramodular forms (analogue of Theorem \ref{Sato-Tate}).


\textbf{Acknowledgments.} We would like to thank K. Morimoto, R. Schmidt and S-W. Shin for helpful discussions.

\section{Genuine forms}\label{smf}
In this section, we follow the notation and the contents in Section 2 of \cite{KWY} for holomorphic Siegel modular forms of genus 2.
The readers should consult the references there if necessary.

\subsection{Classical Siegel modular forms and Hecke fields}\label{class}
For a pair of non-negative integers $\underline{k}=(k_1,k_2)$, $k_1\ge k_2\ge 3$ and $N\ge 1$, we denote by
$S_{\underline{k}}(\G(N),\chi)$ the space of cusp forms of the weight $\underline{k}$ with the character $\chi:(\Z/N\Z)^\times
\lra \C^\times$ for a principal congruence subgroup $\Gamma(N)$. Here the weight corresponds to
the algebraic representation $\lambda_{\underline{k}}$ of $GL_2$ with the highest weight $\underline{k}$ by
$$V_{\underline{k}}={\rm Sym}^{k_1-k_2}{\rm St}_2\otimes {\rm det}^{k_2} {\rm St}_2,
$$
where ${\rm St}_2$ is the standard representation of dimension 2.

For each prime $p\nmid N$ and $n\ge 1$ one can define the Hecke operator $T(p^n)$ acting on $S_{\underline{k}}(\G(N),\chi)$.
There exists a basis of the space consisting of eigenforms for all such $T(p^n)$ which are called Hecke eigen cusp forms.
Let $HE_{\underline{k}}(\G(N),\chi)$ be the set of all Hecke eigen cusp forms for $S_{\underline{k}}(\G(N),\chi)$.
For $F \in HE_{\underline{k}}(\G(N),\chi)$ we have $T(p^n)F=\lambda_F(p^n)F,\ \lambda_F(p^n)\in \C$. Since
$S_{\underline{k}}(\G(N),\chi)$ has an integral structure and is of finite dimensional,
the eigenvalue $\lambda_F(p^n)$ is an algebraic integer once we fix an embedding $\overline{\Q} \hookrightarrow \C$.
We define the Hecke field $\Q_F$ of $F$
by Definition 2.14 in \cite{KWY}.

\subsection{CAP forms and endoscopic lifts}\label{cn}
For $F\in HE_{\underline{k}}(\G(N),\chi)$ let $\pi_F$ be the corresponding cuspidal representation of $GSp_4(\A_\Q)$.
We say $F$ is a CAP form (resp. an endoscopic lift) if $\pi_F$ is a CAP (resp. endoscopic) representation.
We denote by $S_{\underline{k}}(\G(N),\chi)^{{\rm CAP}}$ the space generated by all CAP forms in $S_{\underline{k}}(\G(N),\chi)$.
We also define $S_{\underline{k}}(\G(N),\chi)^{{\rm EN}}$ for endoscopic lifts similarly.
Under mind condition on the level $N$ we have the following estimation on the dimension of the space as above.
\begin{prop}\label{cap} If $k_1\not=k_2$ or $\chi$ is not the square of a character, then $S_{\underline{k}}(\G(N),\chi)^{{\rm CAP}}=0$. If $k:=k_1=k_2$ and $\chi$ is
the square of a character, then
$${\rm dim}S_{\underline{k}}(\G(N),\chi)^{{\rm CAP}}=O(kN^{7+\ve})$$
as $N+k\to \infty$ and $(N,11!)=1$
\end{prop}
\begin{proof} It follows from Section 4.4 and Theorem 4.3 in \cite{KWY}.
\end{proof}

Next we consider endoscopic lifts.

\begin{prop}\label{end} It holds that
$${\rm dim}S_{\underline{k}}(\G(N),\chi)^{{\rm EN}}=O((k_1-k_2+1)(k_1+k_2-3)N^{8+\ve})$$
as $N+k_1+k_2\to \infty$ and $(N,11!)=1$
\end{prop}
\begin{proof}It follows from Theorem 4.2 in \cite{KWY}.
\end{proof}

\subsection{Local-global compatibility for holomorphic Siegel modular forms} In this subsection we
discuss some results which follow from Conjecture 1 of \cite{Jorza} concerning local-global compatibility for
holomorphic Siegel modular forms. That conjecture is satisfied for Siegel modular forms with Iwahori level structure in \cite{Jorza}
and for Siegel paramodular forms by \cite{Sch} and \cite{Sor}.

Let $S^{\text{Stable}}_{\underline{k}}(\G(N))$ be the space generated by all non-endoscopic, non-CAP Hecke eigen cusp forms in
$S_{\underline{k}}(\G(N))$. Such a form is called a stable form by Arthur. For a Hecke eigenform $F$ in
$S^{\text{Stable}}_{\underline{k}}(\G(N))$ let us
consider the corresponding cuspidal representation $\pi_F=\otimes'_{p}\pi_{p}$ of $GSp_4(\A_\Q)$.

\begin{prop}\label{ll} Let the notations be as above. Assume Conjecture 1 of \cite{Jorza}.
Then there exists a unique globally generic cuspidal representation $\pi'=\otimes'_p \pi'_p$ of $GSp_4(\A_\Q)$ such that
$\pi_p$ and $\pi'_p$ belong to the same L-packet and they are both tempered.
\end{prop}
\begin{proof} For such $\pi_F$, by Weissauer \cite{Wei3}, there exists a globally generic cohomological cuspidal representation $\Pi$ of $GSp_4(\A_\Q)$
so that $\pi_F$ is weakly equivalent to $\Pi$. By \cite{Sor}, for any fixed prime $\ell$, it gives rise to a Galois representation
$\rho_{\Pi,\ell}:G_\Q\lra GSp_4(\bQ_\ell)$
which satisfies local-global compatibility. Under Conjecture 1 of \cite{Jorza} we also have
 a Galois representation
$\rho_{\pi_F,\ell}:G_\Q\lra GSp_4(\bQ_\ell)$
which satisfies local-global compatibility and $\rho_{\Pi,\ell}\sim \rho_{\pi,\ell}$ by Chebotarev density theorem.
Therefore $\pi_p$ and $\Pi_p$ belong to the same L-packet. The temperedness of $\Pi_p$ which is proved in \cite{Sor}
implies that of $\pi_p$ by \cite{GT}.
\end{proof}

\begin{cor}\label{packet} Keep the notations as in the previous proposition and assume Conjecture 1 of \cite{Jorza}. Then
$\pi_p$ is generic if its L-packet is a singleton. Otherwise the L-packet of $\pi_p$ has two members which are given in the notation of \cite{RS} by
$\{{\rm VIa, VIb}\}$, $\{{\rm VIIIa, VIIIb}\}$, $\{{\rm Va}, \theta((\sigma {\rm St}_2)^{{\rm JL}},
(\sigma \xi {\rm St}_2)^{{\rm JL}})\}$, 
$\{{\rm XIa}, \theta((\sigma {\rm St})^{{\rm JL}}_2, (\sigma \pi^{{\rm sc}})^{{\rm JL}})\}$, or  
$\{\theta(\pi_{p,1},\pi_{p,2}), \theta((\pi_{p,1})^{{\rm JL}},
(\pi_{p,2})^{{\rm JL}})\}$
 where 
$\pi^{{\rm sc}},\pi_{1,p}$, and $\pi_{2,p}$ stand for unitary supercuspidal representations of $GL_2(\Q_p)$ such that 
$\pi_{1,p}\not\sim \pi_{2,p}$, $\sigma$ is a quasi character of $\Q^\times_p$
and $\xi$ is a quadratic character of $\Q^\times_p$. Here $\theta$ is the local theta correspondence from $GSO(4)$ or $GSO(2,2)$ to $GSp_4$
(cf. Section 4.3 of \cite{KWY}).
\end{cor}
\begin{proof} First we remark that the L-packet of any generic supercuspidal representation which is not in the image of 
the local theta correspondence from $GSO(2,2)$ is a singleton since
it has an irreducible L-parameter. Such a representation never appears in this setting.
Then by using Table A.1 of \cite{RS}, Proposition \ref{ll}, and local Langlands classification \cite{GT}, one can list
every non-generic representation whose L-packet is not a singleton (see also Section 2.4 of \cite{RS}).
\end{proof}

Put $K(p^r):=(1+p^r M_4(\Z_p))\cap GSp_4(\Z_p)$ for a non-negative integer $r$.
\begin{prop}\label{fixed-vector} Let $\pi_p$ be the representation as in Proposition \ref{ll}. 
Assume Conjecture 1 of \cite{Jorza} and that 
the L-packet of $\pi_p$ consists of two members. Assume further that $(p,11!)=1$. Then ${\rm dim}\, \pi^{K(p^r)}_p=O(p^{9r})$ when $p^r$ goes to infinity.
\end{prop}
\begin{proof}Under Conjecture 1 of \cite{Jorza}, one can see that $\pi_p$ is one of representations in Corollary \ref{packet}.
Except for representations obtained by the local theta correspondence it is a constituent of a normalized
induced representation ${\rm Ind}^{GSp_4(\Q_p)}_{B(\Q_p)}\chi$ where $\chi$ is a quasi character of $B(\Q_p)$.
Let $\gamma_1,\ldots, \gamma_t$ be a complete system of the double coset representatives of $B(\Q_p)\backslash GSp_4(\Q_p)/K(p^r)$.
Then $({\rm Ind}^{GSp_4(\Q_p)}_{B(\Q_p)}\chi)^{K(p^r)}$ is generated by the following functions:
$$ f_i(g)=\begin{cases}
\delta^{\frac{1}{2}}_B(b)\chi(b), &\text{if $g=b\gamma_ik \in B(\Q_p)\gamma_i K(p^r)$,} \\
0, & {\rm otherwise.}
\end{cases}
$$
so that $\chi$ is trivial on $B(\Q_p)\cap \gamma_i K(p^r)\gamma^{-1}_i$. 
Therefore we roughly estimate
$${\rm dim}\, \pi^{K(p^r)} \le {\rm dim}\, ({\rm Ind}^{GSp_4(\Q_p)}_{B(\Q_p)}\chi)^{K(p^r)}\le  t=\sharp B(\Z/p^r\Z)\bs GSp_4(\Z/p^r\Z)=
O(p^{4r}).
$$

Let us suppose that
$\pi_p=\theta((\sigma {\rm St}_2)^{{\rm JL}}, (\sigma \xi {\rm St}_2)^{{\rm JL}})$, or
$\pi_p=\theta((\sigma {\rm St})^{{\rm JL}}_2, (\sigma \pi^{{\rm sc}})^{{\rm JL}})$.
Put $\tau=\sigma {\rm St}_2\times \sigma \xi {\rm St}_2$ or $\tau=\sigma {\rm St}_2\times \sigma \pi^{{\rm sc}}$, respectively.
Recall that $(p,11!)=1$.
Then one can apply the argument in Theorem 4.2 of \cite{KWY} which gives us a similar estimation:
$${\rm dim}\, \pi^{K(p^r)}=\frac{{\rm vol}(K_H(p^r))}{{\rm vol}(K(p^r))}\frac{1}{p^{2r}}{\rm dim}\, \tau^{K_{H}(p^r)}=O(p^{9r})$$
where $K_{H}(p^r):={\rm Ker}(H(\Z_p)\lra H(\Z_p/p^r\Z_p))$ for  a unique endoscopic subgroup $H=GSO(2,2)$ in $GSp_4$ (see Section 4.3 of \cite{KWY}). The remaining case is similar to this case. 
\end{proof}

The following lemma will be used to define a non-canonical map (\ref{gg}). 

\begin{lem}\label{level-preserve} Let $\pi_p$ be as in Corollary \ref{packet}. Assume that L-packet of $\pi_p$ is not a singleton and 
let $\{\pi_p,\pi^g_p\}$ be the L-packet where $\pi_p$ is non-generic and $\pi^g_p$ is generic. Then for a positive integer $r$,  
$(\pi^g_p)^{K(p^r)}\not=0$ (resp. $(\pi^g_p)^{K(p^{\max\{4,r\}})}\not=0$)  if $(\pi_p)^{K(p^r)}\not=0$ except for 
the case of $\pi_p=VIb$ (resp. for the case of $\pi_p=VIb$).
\end{lem}
\begin{proof} We prove this lemma case by case and all cases are given in Corollary \ref{packet}. 
Let us first consider the type VIa and VIb. Then $\pi^g_p=\tau(S,\nu^{-\frac{1}{2}}\sigma)$ and $\pi_p=\tau(T,\nu^{-\frac{1}{2}}\sigma)$.  
By definition $\tau(T,\nu^{-\frac{1}{2}}\sigma)=\sigma\otimes \tau(T,\nu^{-\frac{1}{2}})$ and 
$\tau(S,\nu^{-\frac{1}{2}}\sigma)=\sigma\otimes \tau(S,\nu^{-\frac{1}{2}})$. Therefore the character 
$\sigma$ should be trivial on $1+p^r\Z_p$. Note that $\tau(S,\nu^{-\frac{1}{2}})$ (resp. 
$\tau(S,\nu^{-\frac{1}{2}})$) is of paramodular level 4 (resp. 2). Since the paramodular subgroup $K^{{\rm para}}(p^t)$ contains 
$K(p^t)$ for any positive integer $t$, the claim follows. 

In the case of VIIIa and VIIIb, by Table A.4 of \cite{RS}, both of their Jacquet modules along the Klingen parabolic subgroup are 
$1\rtimes \pi$ where $\pi$ is a unitary supercuspidal representation of $GL_2(\Q_p)$. The claim follows from this. 

Next we consider the case $\pi^g_p=Va=\delta([\xi,\nu\xi],\nu^{-\frac{1}{2}}\sigma)$ where $\xi$ is a quadratic character. 
Clearly $\sigma$ is trivial on $1+p^r\Z_p$. 
If $\xi \sigma$ is ramified, then $\delta([\xi,\nu\xi],\nu^{-\frac{1}{2}}\sigma)=\sigma \otimes 
\delta([\xi,\nu\xi],\nu^{-\frac{1}{2}})$. The claim is now easy to follow. The other remaining case is done similarly. 

For the level of theta correspondence we can apply the argument in Theorem 4.2 of \cite{KWY} regarding 
the transfer of some Hecke elements with respect to congruence subgroups and the claim follows from this directly. 
\end{proof}

\subsection{A non-canonical map between stable forms} Throughout this subsection we assume Conjecture 1 of \cite{Jorza}
as in Proposition \ref{ll} and Arthur's classification for $GSp_4$ (which will be completed soon). Under these assumptions we relate holomorphic stable forms with
globally generic stable forms. According to Lemma \ref{level-preserve} we also assume that ${\rm ord}_p(N)\ge 4$ when a prime $p$ 
divides $N$.  
Let $\pi=\otimes_p \pi_p$ be a cuspidal representation which comes from a Hecke eigen form in $S^{\text{Stable}}_{\underline{k}}(\G(N))$.
Let $\pi^g_\infty:=D^{{\rm large}}_{l_1,-l_2}$ be the large discrete series with Harish-Chandra parameter $(l_1,l_2)=(k_1-1,k_2-2)$
(see Section 2.3 of \cite{Wakatsuki1})
and we denote by $S^{{\rm large}}_{\underline{k}}(\G(N))$ the space of $C^\infty$ Hecke eigen automorphic forms
whose representation of $GSp_4(\R)$ is isomorphic to $D^{{\rm large}}_{l_1,l_2}$ and which
give rise to globally generic representations. We say an element of this space
a globally generic form. Let $S^{{\rm gg}}_{\underline{k}}(\G(N))$ be the space of globally generic forms.
For $\pi=\otimes_p \pi_p\in S^{\text{Stable}}_{\underline{k}}(\G(N))$, we define the finite set $S$ of primes $p$ so that the L-packet of $\pi_p$ is not a singleton and $\pi_p$ is
non-generic. We denote by $\pi^g_p$ the generic representation in the same L-packet as $\pi_p$ for $p\in S$.
Since $\pi$ is stable, by Arthur's classification for $GSp_4$ (cf. line 5 in p.11 of \cite{Sch-P}), the admissible representation
$$\pi^g:=D^{{\rm large}}_{l_1,-l_2}\otimes \bigotimes_{p\in S}\pi^g_p \otimes \bigotimes_{p\not\in S}\pi_p$$
is an automorphic cuspidal representation which is generic everywhere. Then by Proposition \ref{ll},
this is a globally generic representation so there is a unique distinguished vector $F^g$ in $\pi^g$.
The uniqueness follows from \cite{JS}. 
Therefore we have a non-canonical map
\begin{equation}\label{gg}
T^g:S^{\text{Stable}}_{\underline{k}}(\G(N))\lra S^{{\rm gg}}_{\underline{k}}(\G(N)),\ F\mapsto F^g.
\end{equation} 
Note that $T^g$ takes the forms with respect to $\G(N)$ by Lemma \ref{level-preserve}. 
For any subset $S'\subset S\cup \{\infty\}$ one can also consider 
$\pi(S'):=\bigotimes_{p\in S}\pi^g_p \otimes \bigotimes_{p\not\in S'}\pi_p$ which is also 
automorphic. Similarly we have a non-canonical map 
$T(S')$ from $S^{\text{Stable}}_{\underline{k}}(\G(N))$. There exists $2^{|S|}$ maps for $T(S')$. 

We will exploit this map to estimate the number of functorial lifts with the following proposition.
\begin{prop} Under the assumptions in this section, ${\rm dim}\, {\rm Ker}T^g=O(N^{8+\ve})$.
\end{prop}
\begin{proof}
Let $\mathcal{A}^{\text{Stable}}_{\underline{k}}(\G(N))$ be the space of automorphic forms on $GSp_4(\A)$ which 
corresponds to $S^{\text{Stable}}_{\underline{k}}(\G(N))$.  
Then by \cite[(2.17)]{KWY},
${\rm dim}\mathcal{A}^{\text{Stable}}_{\underline{k}}(\G(N))=\varphi(N){\rm dim }S^{\text{Stable}}_{\underline{k}}(\G(N))$. 
Hence the contribution from adelic forms to classical forms is differ by $\varphi(N)$. 
Proposition \ref{fixed-vector} implies that the dimension of the adelic version of $T^g$ is $O(N^9)$.
Since $2^{|S|}=O(2^{\omega(N)})=O(N^\ve)$, the claim follows. 
\end{proof}

\subsection{Symmetric cube lifts}\label{cubic}
There is a functorial lift from $GL_2$ to $GL_4$ which is called symmetric cube lift constructed by Kim and Shahidi
\cite{KS}. For an elliptic cusp form $f$ of weight greater than 2, let $\Pi'$ be the image of $\pi_f$ under the symmetric cube lift.
Then it descends to a unique globally generic cuspidal representation $\Pi_f:={\rm Sym^3 \pi_f}$ of $GSp_4(\A_\Q)$.
We say $F\in HE_{\underline{k}}(\G(N))$ is a symmetric cube lift if $\pi_{F,p}\sim \Pi_{f,p}$ for almost all prime $p\not=\infty$ and
an elliptic cusp form $f$.
We denote by $S_{\underline{k}}(\G(N),\chi)^{{\rm Cube}}$ the space generated by all symmetric cube lifts in
$S_{\underline{k}}(\G(N),\chi)$.

In what follows we will try to estimate the dimension of this space.
We can define the following map as in (\ref{gg}):
\begin{equation}\label{T}
T^g:=T^g_{\underline{k}}(N,\chi):S^{\text{Stable}}_{\underline{k}}(\G(N),\chi)\lra
S^{\text{non-E,gg}}_{\underline{k}}(\G(N),\chi),\
T(F)=F^g
\end{equation}
where $S^{\text{Stable}}_{\underline{k}}(\G(N),\chi)$ is the space generated by all Hecke eigen stable forms in
$S_{\underline{k}}(\G(N),\chi)$ and
$S^{\text{non-E,gg}}_{\underline{k}}(\G(N),\chi)$ is the space generated by all non-endoscopic globally generic Hecke eigen forms in
$S^{{\rm gg}}_{\underline{k}}(\G(N),\chi)$.

Let $S^{{\rm Cube}}_{\underline{k}}(\G(N),\chi)$ (resp.
$S^{{\rm Cube,\ gg}}_{\underline{k}}(\G(N),\chi)$) be the space generated by all symmetric
cube lifts (resp. all generic symmetric cube lifts) in
$S^{{\rm Cube}}_{\underline{k}}(\G(N),\chi)$ (resp.
$S^{{\rm gg}}_{\underline{k}}(\G(N),\chi)$). It is easy to see that
$S^{{\rm Cube}}_{\underline{k}}(\G(N),\chi)\subset S^{\text{Stable}}_{\underline{k}}(\G(N),\chi)$ and
$S^{{\rm Cube,\ gg}}_{\underline{k}}(\G(N),\chi)\subset S^{\text{non-E,gg}}_{\underline{k}}(\G(N),\chi)$.
Furthermore $T^g$ induces a surjective linear map
$$T^g:S^{{\rm Cube}}_{\underline{k}}(\G(N),\chi) \lra
S^{\text{Cube,gg}}_{\underline{k}}(\G(N),\chi).$$
Then we have
\begin{thm}Fix a weight $\underline{k}$. Then it holds that
$${\rm dim}S^{\text{Cube,gg}}_{\underline{k}}(\G(N),\chi)=O(N^{4+\ve})$$ as $N\to \infty$ with $(N,11!)=1$.

\end{thm}
\begin{proof}
Let $F$ be a globally generic Hecke eigen cusp form in $S^{\text{Cube,gg}}_{\underline{k}}(\G(N),\chi)$.
By Lemma 9.1 of \cite{KWY}, the conductor of $\pi_F$ is bounded by $N^4$.
On the other hand, if $\pi_F={\rm Sym^3}\pi_f$ for an elliptic cusp form $f$, by applying
Theorem in Section 6.5 of \cite{BHK} to ${\rm Sym}^2\pi_f$ and $\pi_f$, the lower bound of the conductor of $\pi_F$
is given by
$$\frac{c(\pi_f)(c(\pi_f)^2-c(\pi_f))}{c(\pi_f)}=c(\pi_f)^2-c(\pi_f)$$
since ${\rm Sym}^2 \pi_f\boxtimes \pi_f={\rm Sym}^3\pi_f\boxplus \pi_f\otimes \omega_{\pi_f}$.
Hence we have that $c(\pi_f)=O(N^2)$. The first claim follows from this with
the dimension formula for the space of elliptic cusp forms with respect to $\Gamma_1(N^2)$ for
a fixed weight.
\end{proof}

\subsection{Automorphic induction}\label{ai}
In this section we are concerned with Automorphic induction. For a quadratic field $K/\Q$,
there is a functorial lift from $GL_2/K$ to $GL_4/\Q$ which is called Automorphic induction (some of people say
an Asai lift). To descend it to a globally generic representation of $GSp_4$,
the central character $\omega_\pi$ should be invariant under the non-trivial element $\sigma$ in ${\rm Gal}(K/\Q)$.
Hence $\omega_\pi=\omega\circ N_{K/\Q}$ for a character $\omega:\Q^\times\backslash \A^\times \lra \C^\times$.

When $K$ is a real quadratic field, then for any Hilbert modular cusp form $f$ of weight $(k_1+k_2-2,k_1-k_2+2)$
one can construct a generic cuspidal automorphic representation whose representation $\Pi$ at infinity is
isomorphic to $D^{{\rm large}}_{l_1,l_2}$ with $(l_1,l_2)=(k_1-1,k_2-2)$.
We say $F\in S_{\k}(\G(N))$ is an automorphic induction if $\pi_F$ is weakly equivalent to such a $\Pi$.

Similarly one can consider an automorphic induction for an imaginary quadratic field $K$.
However by \cite{HST} if  $F\in S_{\k}(\G(N))$ is given by such a way, then
$k_2$ has to be $2$ which contradicts with our condition $k_2\ge 3$.
Therefore we have only to consider only contributions from Hilbert modular forms for
real quadratic fields.

Let $S^{{\rm AI}}_{\underline{k}}(\G(N),\chi)$ (resp.
$S^{{\rm AI, gg}}_{\underline{k}}(\G(N),\chi)$) be the space generated by all automorphic induction
(resp. all generic automorphic induction) in
$S_{\underline{k}}(\G(N),\chi)$ (resp.
$S^{{\rm gg}}_{\underline{k}}(\G(N),\chi)$). It is easy to see that
$S^{{\rm AI}}_{\underline{k}}(\G(N),\chi)\subset S^{\text{stable}}_{\underline{k}}(\G(N),\chi)$ and
$S^{{\rm AI,\ gg}}_{\underline{k}}(\G(N),\chi)\subset S^{\text{gg,non-E}}_{\underline{k}}(\G(N),\chi)$.
Furthermore $T^g$ induces a surjective linear map
$$T^g:S^{{\rm AI}}_{\underline{k}}(\G(N),\chi) \lra
S^{\text{AI,gg}}_{\underline{k}}(\G(N),\chi).$$
Then we have
\begin{thm}Fix a weight $\underline{k}$. Then it holds that
$${\rm dim}S^{{\rm AI,gg}}_{\underline{k}}(\G(N),\chi)=O(N^{\frac{11}{2}+\ve})$$ as $N\to \infty$ with $(N,11!)=1$.

\end{thm}
\begin{proof}
Let $F$ be a globally generic Hecke eigen cusp form in $S^{\text{AI,large}}_{\underline{k}}(\G(N),\chi)$.
By Lemma 9.1 of \cite{KWY} again the conductor of $\pi_F$ is  bounded by $N^4$.
Assume that $\pi_F$ comes from a unique Hilbert cusp form for a quadratic field $K$ with
the weight $(k_1+k_2-2,k_1-k_2+2)$ and the level  $\mathcal{A}\subset \mathcal{O}_K$ with a character $\chi'$ so that
$\chi'{}^\sigma \chi'=\chi$.
By comparing the conductor we have $D_KN_{K/\Q}(\mathcal{A})=O(N^4)$
Then we have
$${\rm dim}S^{{\rm AI,large}}_{\underline{k}}(\G(N),\chi)=O(N^4\sum_{D_K|N \atop
K:\text{real quadratic}}\zeta_K(-1))$$
by Shimizu's dimension formula in \cite{Shimizu}.

Now for a real quadratic field $K$, $\zeta_K(-1)=\ds\frac {\zeta_K(2)}{\pi^3\Gamma(-\frac 12)^2} D_K^{\frac 32}$ by the functional equation.
Since $\zeta_K(2)\leq \zeta(2)^2$,
$$\sum_{D_K|N \atop
K:\text{real quadratic}}\zeta_K(-1))\ll \sum_{D_K|N} D_K^{\frac 32}\ll N^{\frac 32+\epsilon}.
$$
\end{proof}

\subsection{Asai transfer}\label{asai}
In this section we are concerned with the Asai transfer which is a transfer from $GL_2/K$ to $GL_4/\Q$ for an etale algebra $K$
of degree 2 over $\Q$.
When $F=\Q\oplus \Q$, it is the Rankin-Selberg convolution product. When $F$ is a field, it is called the Asai transfer.
In the former case, given a pair $(\pi_1,\pi_2)$ of two cuspidal representations of $GL_2(\A_\Q)$, the automorphic product $\pi_1\boxtimes\pi_2$ descend to a unique globally generic representation of $GSp_4$ only when one of $\pi_i$'s should be dihedral. (See \cite{Ki}.)
When $K$ is a field, the Asai transfer $As(\pi)$ of a cuspidal representation $\pi$ of $GL_2/K$ descend to $GSp_4$ only when $\pi$ is
 dihedral.
We say $F\in S_{\k}(\G(N),\chi)$ an Asai lift if $\pi_F$ is weakly equivalent to such a representation of $GSp_4$.

We denote by $S^{\text{Asai, non-E}}_{\underline{k}}(\G(N),\chi)$ the space generated by all non-endoscopic Asai lifts in
$S^{\text{stable}}_{\underline{k}}(\G(N),\chi)$.
Similarly we can define  $S^{\text{Asai, non-E, gg}}_{\underline{k}}(\G(N),\chi)$.
As in the previous section we have a surjective linear map
$$T^g:S^{{\rm Asai, non-E}}_{\underline{k}}(\G(N),\chi)\lra S^{\text{Asai, non-E, gg}}_{\underline{k}}(\G(N),\chi).$$
By using this map we have
\begin{thm}Fix a weight $\underline{k}$. Then it holds that
$${\rm dim} \, S^{\text{Asai,non-E,gg}}_{\underline{k}}(\G(N),\chi)=O(N^{6+\ve})$$
as $N\to \infty$ with $(N,11!)=1$.
\end{thm}
\begin{proof} Let us first consider the case when $K$ is not a field.
Let $(f_1,f_2)$ be a pair of two elliptic cusp forms of level $N_1,N_2$ with characters $\chi_1,\chi_2$ so that
$\chi_1\chi_2=\chi$ but one of them is a CM form (say, $f_2$). Note that once we fix an imaginary quadratic field defining
$f_2$, then the character of $f_2$ is fixed and so is $f_1$.
As in the previous proposition, we have $N_1N_2\le N^4$ by Theorem in Section 6.5 of \cite{BHK}. Therefore we have only to
consider
$$\sum_{N_1|N^4}\sum_{a^2D_K|N_1\atop K:\text{imaginary quadratic}}h(a^2 D_K)\cdot {\rm dim}S^1_{k_1}(\G_0(N^4N_1^{-1}),\ast),
$$
where $a^2$ is the square factor dividing $N_1$, and $\ast$ is a fixed character and $h(a^2D_K)$ stands for the class number of the binary forms with discriminant $a^2D_K$.

Now for a fixed $k_1$ and $N_1$, ${\rm dim}S^1_{k_1}(\G_0(N^4N_1^{-1})=O(N^4N_1^{-1})$, and
$$h(a^2D_K)=h_K a \prod_{p|a} (1-(\frac {D_K}p)p^{-1})\ll |D_K|^{\frac 12+\epsilon} a^{1+\epsilon}.
$$
Hence
$${\rm dim} \, S^{\text{Asai, non-E, gg}}_{\underline{k}}(\G(N),\chi)\ll
\sum_{N_1|N^4} N^4N_1^{-1}\sum_{a^2d|N_1} (a^2d)^{\frac{1}{2}+\ve} \ll N^4\sum_{N_1|N^4} N_1^{-\frac 12+\epsilon}\ll N^{4+\epsilon}.
$$

When $K$ is a field, then there exists a quartic field $L/\Q$ which contains $K$ such that $\pi_F$ is
weakly equivalent to $AI^L_\Q(\tau)$ for some Hecke character $\tau:\A^\times_L/L^\times \lra \C^\times$.
By comparing the conductor we see that $D_L N_{L/\Q}(c(\tau))=O(N^4)$ where $c(\tau)$ is the conductor of $\tau$.
Hence we need to count the number of such Hecke characters.
Given a quartic field of discriminant $d$, $N_{L/\Q}(c(\tau))=O(N^4 |d|^{-1})$.
Given a positive integer $n$, there are $O(\tau(n))$ integral ideals of norm $n$, where $\tau(n)$ is the number of divisors of $n$.
Given an ideal $\frak a$, there are at most $N(\frak a)h_L$ Hecke characters with the conductor $\frak a$.
Therefore, given a positive integer $n$, there are at most $n\tau(n)h_L\ll n^{1+\epsilon}D_L^{\frac 12+\epsilon}$ Hecke characters.
Hence the number of Hecke characters in question is
$$\ll \sum_{|D_L|<N^4} (N^4 |d_L|^{-1})^{1+\epsilon}D_L^{\frac 12+\epsilon}\ll N^{4+\epsilon}\sum_{|D_L|<N^4} |D_L|^{-\frac 12+\epsilon}.
$$
Let $N(d)$ be the number of quartic fields of discriminant $|d|$ which contains a quadratic subfield. Then $\sum_{|d|\leq x} N(d)\ll x$ (\cite{Co}). Hence by partial summation,
the last estimation  becomes $O(N^{6+\epsilon})$.
Hence our result follows.
\end{proof}

\subsection{Proofs of main theorems for Hecke fields}
We are now ready to prove main theorems. Before going into proofs we give a precise definition of genuine forms.
\begin{Def}\label{ge} Let $F$ be a Hecke eigen Siegel cusp form of weight $(k_1,k_2),\ k_1\ge k_2\ge 3$ which is neither a CAP form nor an endoscopic lift. We say $F$ is a genuine form if the
corresponding automorphic representation $\pi_F$ is not weakly equivalent to any of a base change lift, an Asai lift, and a symmetric cube lift.
\end{Def}
By using results in the previous section, Theorem 1.1 and 1.3 of \cite{KWY} hold if we replace
$HE_{\k}(N,\chi)$ with the subset $\{F\in HE_{\k}(N,\chi)\ |\ F:\text{genuine}\}$, since
\begin{eqnarray*}
&& \{F\in HE_{\k}(N,\chi)\ |\ F:\text{non-genuine}\} =
O({\rm dim\, Ker}(T^g))+{\rm dim} \, S^{{\rm Cube, gg}}_{\underline{k}}(\G(N),\chi) \\
&& \phantom{xxxxxxxxxxxxxxxxx} +{\rm dim} \, S^{{\rm AI,gg}}_{\underline{k}}(\G(N),\chi)+
{\rm dim} \, S^{\text{Asai,non-E,gg}}_{\underline{k}}(\G(N),\chi)= O(N^{8+\ve}).
\end{eqnarray*}
This proves Theorem \ref{hecke1}.

Theorem \ref{hecke2},\ref{hecke3} follow from the argument in Section 6.2 and 6.3 of \cite{ST1} with
Theorem 1.1 of \cite{KWY}.

\section{Galois representations for genuine forms}
In this section we characterize a genuine form in terms of Galois representations.
Let $F$ be a Hecke eigen Siegel cusp form in $S_{\k}(\G(N))$ which
is neither CAP nor endoscopic. By Laumon-Weissauer, for any prime $\ell$ there
exists a unique irreducible Galois representation
$$\rho_{F,\ell}:G_\Q:={\rm Gal}(\bQ/\Q)\lra {\rm GSp}_4(\bQ_\ell)$$
such that
$\det(I_4-X\rho_{F,\ell}({\rm Frob}_p))$ coincides with the Hecke polynomial at $p$ for any prime $p\nmid \ell N$
(see (2.13) of \cite{KWY} for Hecke polynomials).
Since $\pi_F$ is weakly equivalent to a generic cuspidal representation of $GSp_4$, it can be transfered to
a cuspidal representation of $GL_4$. The irreducibility of $\rho_{F,\ell}$ follows from this fact and
the main result of \cite{CG}.

\begin{thm}The notation being as above.
Then $F$ is genuine if and only if the Zariski closure of the image of $\rho_{F,\ell}$ contains $Sp_4$ for all but
finitely many $\ell$.
\end{thm}
\begin{proof}
Let $G_\ell$ be the Zariski closure of the image of $\rho_{F,\ell}$.
By Corollary 4.4 and Proposition 4.5 of \cite{CG} one of the following cases happens:
(1) $G_\ell$ contains $Sp_4$;
(2) $G_\ell$ is contained in the subgroup $GSO(2,2)=GL_2\times GL_2/GL_1$ of $GSp_4$;
(3) $G_\ell$ is contained in the subgroup ${\rm Sym}^3GL_2$ of $GL_4$;
(4) $\rho_{F,\ell}$ is an induced representation of 2-dimensional Galois representation associated to a
Hilbert modular form for a quadratic real field;
(5)  $\rho_{F,\ell}$ is an induced representation of a character.

By Theorem 8.7 for all but finitely many $\ell$, the reduction $\br_{F,\ell}$ is irreducible.
In the case (2), there exists a 2-dimensional irreducible Galois representation $\rho_\ell:G_\Q\lra GL_2(\bQ_\ell)$
such that $\rho_{F,\ell}={\rm Sym}^3\rho_\ell$. The irreducibility of $\br_{F,\ell}$ implies that of $\br_\ell$.
Therefore $\br_\ell$ is modular by Serre conjecture due to Khare-Wintenberger \cite{KW}.  It follows that its image
contains $SL_2(\mathbb{F}_\ell)$ otherwise ${\rm Sym}^3\rho_\ell$ is reducible. Hence $\rho$ is absolute irreducible even if
we restrict it to $G_{\Q(\zeta_\ell)}$.
Then by Kisin \cite{Kisin} for
the ordinary case and Emerton \cite{Emerton} for the non-ordinary case,
$\rho_p$ is modular. Hence the case (2) corresponds to a symmetric cubic lift.

We now assume that $F$ is genuine and $\rho_{F,\ell}$ does not contain $Sp_4$ for infinitely many $\ell$.
We denote by $\Sigma$ the set of such primes $\ell$.
The cases (3), (4) and (5) are excluded since $F$ is genuine with the argument above in the case (3).
Then by applying Theorem 1.0.1 of \cite{LY} for a sufficiently large $\ell\in \Sigma$ (note that they used the notation $SGO(4)$ instead of $GSO(4)$), we see that $F$ is a Rankin-Selberg convolution which contradicts with the assumption on $F$.
\end{proof}

\section{Simultaneous vertical Sato-Tate theorem}\label{sato-tate}

Let $F\in HE_{\underline{k}}(\Gamma(N),\chi)^{\rm tm}$.
For a prime $p\nmid N$, let $a_{F,p}, b_{F,p}\in [-2,2]$ be Hecke eigenvalues as in \cite{KWY}. Recall also the measure
$$\mu_p=f_p(x,y)g^+_p(x,y)g^-_p(x,y)\cdot \mu^{{\rm ST}}_{\infty}$$ on
$\Omega:=[-2,2]^2/\mathfrak{S}_2$ (the non-trivial element in $\mathfrak{S}_2$ acts by $(x,y)\mapsto (y,x)$ on $[-2,2]^2$), where
$$f_p(x,y)=\frac{(p+1)^2}{\left(\left(\sqrt{p}+\frac{1}{\sqrt{p}}\right)^2-x^2\right)
\left(\left(\sqrt{p}+\frac{1}{\sqrt{p}}\right)^2-y^2\right)},\
 \mu^{{\rm ST}}_{\infty}=\frac{(x-y)^2}{\pi^2}\sqrt{1-\frac{x^2}{4}}\sqrt{1-\frac{y^2}{4}},$$
$$g^{\pm}_p(x,y)=\frac{p+1}{\left(\sqrt{p}+\frac{1}{\sqrt{p}}\right)^2-2\left(1+\frac{xy}{4}\pm\sqrt{1-\frac{x^2}{4}}\sqrt{1-\frac{y^2}{4}}\right)}.
$$
Let $C^0(\Omega,\R)$ be the space of $\R$-valued continuous functions on $\Omega$.
Then we can generalize Theorem 1.4 of \cite{KWY} to finitely many primes.

\begin{thm}\label{Sato-Tate}
Let $p_1,...,p_r$ be distinct primes. Then
for $f\in C^0(\Omega^r,\Bbb R)$,
\begin{eqnarray*}
&&\frac 1{d^{\rm tm}_{\underline{k},N}(\chi)}\sum_{F\in HE_{\underline{k}}(\G(N),\chi)} f((a_{F,p_1},b_{F,p_1}),...,(a_{F,p_r},b_{F,p_r})) =\int_{\Omega^r} f((x_1,y_1),...,(x_r,y_r))\prod_{i=1}^r d\mu_{p_i} \\
&&\phantom{xxxxxxxxx} + \begin{cases} O((p_1\cdots p_r)^a \phi(N) N^{-2}), &\text{level aspect}\\
O\left(\frac {(p_1\cdots p_r)^b}{((k_1-k_2+1)(k_1-1)(k_2-2)}\right) + O\left(\frac {(p_1\cdots p_r)^c}{((k_1-k_2+1)(k_1+k_2-3)}\right), &\text{weight aspect}\end{cases}
\end{eqnarray*}
for some constants $a,b,c>0$.
\end{thm}
\begin{proof} Take $S'=\{p_1,...,p_r\}$ in Proposition 5.3 of \cite{KWY}. Then we can follow along the proof of Theorem 1.3 of \cite{KWY} by using
Theorem 6.4 and 6.5 of \cite{KWY}.
\end{proof}

\section{$n$-level density of degree 4 spinor $L$-functions}\label{spinor}
Katz and Sarnak \cite{KS1} proposed a conjecture on low-lying zeros of $L$-functions in natural families, which says that the distributions of the low-lying zeros of $L$-functions in a family $\frak{F}$ is predicted by a symmetry type $G(\frak{F})$ attached to
$\frak{F}$: For a given entire $L$-function $L(s,\pi)$, we denote the non-trivial zeros of $L(s,\pi)$ by $\frac 12 + \sqrt{-1}\gamma_j$. Since we don't assume GRH for $L(s,\pi)$, $\gamma_j$ can be a complex number. Let $\phi(x_1,...,x_n)=\prod_{i=1}^n \phi_i(x_i)$ be an even Schwartz class function in each variables whose Fourier transform $\hat\phi(u_1,...,u_n)$
is compactly supported. We define
\begin{equation*}
D^{(n)}(\pi, \phi)
={\sum}_{j_1,\cdots,j_n}^*\phi\left(\gamma_{j_1}\frac{\log c_\pi}{2 \pi},\gamma_{j_2}\frac{\log c_\pi}{2 \pi},\dots,\gamma_{j_n}\frac{\log c_\pi}{2 \pi}\right)
\end{equation*}
where $\sum_{j_1,...,j_n}^*$ is over $j_i \in \Z$ (if the root number is $-1$) or $\Z  \backslash \{0\}$  with $j_{a}\ne \pm j_{b}$ for $a\ne b$, and $c_{\pi}$ is the analytic conductor of $L(s,\pi)$.

Let $\frak{F}(X)$ be the set of $L$-functions in $\frak F$ such that $X< c_{\pi}< 2X$. The $n$-level density conjecture says that
\begin{equation*}
\lim_{ X\to \infty} \frac{1}{|\frak{F}(X)|} \sum_{\pi \in \frak{F}(X)}D^{(n)}(\pi,\phi)=
\int_{\Bbb R^n} \phi(x)W(G(\frak{F}))\, dx,
\end{equation*}
where $W(G(\frak{F}))$ is the $n$-level density function.

There are five possible symmetry types of families of $L$-functions: U, SO(even), SO(odd), O, and Sp. The corresponding density functions
$W(G)$ are determined in \cite{KS1} (cf. \cite{Rub}). They are
\begin{eqnarray*}
&&W(\text{U})(x)=\text{det}(K_0(x_j,x_k))_{1\leq j\leq n\atop 1\leq k\leq n},\\
&& W(\text{SO(even)})(x) = \text{det}(K_1(x_j,x_k))_{1\leq j\leq n\atop 1\leq k\leq n},\\
&&W(\text{SO(odd)})(x)= \text{det}(K_{-1}(x_j,x_k))_{1\leq j\leq n\atop 1\leq k\leq n}+\sum_{\nu=1}^n \delta(x_\nu) \text{det}(K_{-1}(x_j,x_k))_{1\leq j\ne\nu\leq n\atop 1\leq k\ne \nu\leq n}, \\
&&W(\text{Sp})(x) = \text{det}(K_{-1}(x_j,x_k))_{1\leq j\leq n\atop 1\leq k\leq n},\quad W(\text{O})(x)=\frac 12 (W(\text{SO(even)})(x)+W(\text{SO(odd)})(x)),
\end{eqnarray*}
where $K_\epsilon(x,y)=\dfrac {\sin \pi(x-y)}{\pi(x-y)}+\epsilon \dfrac {\sin \pi(x+y)}{\pi(x+y)}$, $\epsilon\in \{\pm 1, 0\}$.

We will study the family of degree 4 spinor $L$-functions $L(s,\pi_F,{\rm Spin})$ for $F\in HE_{\underline{k}}(N)$.
We may assume that $\pi_F$ is not a CAP form.
(For a CAP form, $|a_F(p)|\leq 4p^{\frac 12}$ and $|a_F(p^2)|\leq 4p$ in (\ref{explicit}). Hence if the support of $\phi$ is smaller than $(-1,1)$, then the sum over $p$ in (\ref{explicit}) is $O(N)$. But the dimension of the space of CAP forms is $O(N^{7+\epsilon})$. Hence it is negligible.)

Let $L(s,\pi_F, \text{Spin})$ be the degree 4 spinor $L$-function. Let
$$L(s,\pi_F,\text{Spin})=\sum_{n=1}^\infty \tilde\lambda_F(n) n^{-s}.
$$
It satisfies the functional equation:
$$\Lambda(s,\pi_F,{\rm Spin})=q(F)^{\frac s2} \Gamma_\Bbb C(s+\frac {k_1+k_2-3}2)\Gamma_\Bbb C(s+\frac {k_1-k_2+1}2)L(s,\pi_F,{\rm Spin}),
$$
$$\Lambda(s,\pi_F,{\rm Spin})=\epsilon(\pi_F)\Lambda(1-s,\pi_F,{\rm Spin}),
$$
where $\epsilon(\pi_F)\in\{\pm 1\}$ and $N\le q(F)\le N^4$.

Let $\phi$ be a Schwartz function which is even and whose Fourier transform has a compact support.
Define
\begin{equation*}
D(\pi_F,\phi,{\rm Spin}) = \sum_{\gamma_{F}}\phi\left( \frac{\gamma_{F}}{2\pi} \log c_{\underline{k},N}\right),
\end{equation*}
where $\log c_{\underline{k},N}=\frac 1{d_{\underline{k},N}} \sum_{F\in HE_{\underline{k}}(N)} \log c(F)$ for $\underline{k}=(k_1,k_2)$, and $c(F)=(k_1+k_2)^2(k_1-k_2+1)^2 q(F)$ is the analytic conductor. We showed in \cite{KWY} that

$$\frac 1{d_{\underline{k},N}} \sum_{F\in HE_{\underline{k}}(N)}  D(\pi_F,\phi,{\rm Spin})=\widehat{\phi}(0)+\frac 12 \phi(0)
+O\left( \frac{\omega(N)}{\log c_{\underline{k},N}} \right),
$$
where $\omega(N)$ is the number of distinct prime factors of $N$.
So the possible symmetry type could be O, SO(even) or SO(odd) but we cannot distinguish them because the support of $\hat{\phi}$ is too small. In order to distinguish them, we need to compute the $n$-level density.

Let $\lambda_F(p^n)$ be the eigenvalue of the Hecke operator $T(p^n)$ for $p\nmid N$. Here
$T(p^n)$ is the sum of Hecke operators $T_m$ for $m={\rm text}(p^{a_1},p^{a_2},p^{-a_1+\kappa},p^{-a_2+\kappa})$, where $0\leq a_2\leq a_1\leq \kappa$.

Let $T'(p^n)=T(p^n)p^{-n\frac {k_1+k_2-3}2}$. Then
$$L(s,\pi_F,\text{Spin})=L(s-\tfrac {k_1+k_2-3}2,F,\text{Spin}),
$$
and $L(s,F,\text{Spin})=\prod Q_{F,p}(p^{-s})^{-1}$,
$$Q_{F,p}(t)^{-1}=(1-p^{k_1+k_2-4}t^2)^{-1}\sum_{n=0}^\infty \lambda_F(p^n) t^n.
$$

Therefore
$$L(s,\pi_F,\text{Spin})=\prod Q'_{F,p}(p^{-s})^{-1},
$$
where
$$Q'_{F,p}(t)^{-1}=(1-p^{-1}t^2)^{-1}\sum_{n=0}^\infty \lambda'_F(p^n) t^n.
$$

Hence $\tilde\lambda_F'(p)=\lambda'_F(p)$, and
$$\tilde\lambda_F'(p^n)=\begin{cases} p^{-\frac n2}+\displaystyle\sum_{i=1}^{\frac n2} p^{-\frac n2+i}\lambda_F'(p^{2i}), &\text{if $n$ is even}\\
\displaystyle\sum_{i=0}^{\frac {n-1}2} p^{-\frac {n-1}2+i}\lambda_F'(p^{2i+1}), &\text{if $n$ is odd}\\
\end{cases}
$$

Therefore by \cite[Theorem 8.1]{KWY}, we have the following. Note that $T'(p^{2i})$ is a sum of Hecke operators and only $T'_{p^iE_4}$ contributes the main term $p^{-3i}$.

\begin{thm}\label{spin}
Put $\underline{k}=(k_1,k_2),\ k_1\ge k_2\ge 3$.
\begin{enumerate}
\item (level-aspect) Fix $k_1,k_2$. Then as $N\to\infty$,
$$\frac{1}{d_{\underline{k},N}}\sum_{F\in HE_{\underline{k}}(N)} \tilde\lambda_F(p^n) = \begin{cases} \sum_{i=0}^{\frac n2} p^{-\frac n2-2i}+O(N^{-2}p^c), &\text{if $n$ is even}\\
O(N^{-2}p^c), &\text{if $n$ is odd}\end{cases},
$$
for some constant $c>0$.
\item (weight-aspect) Fix $N$. Then as $k_1+k_2\to\infty$,
$$\frac{1}{d_{\underline{k},N}}\sum_{F\in HE_{\underline{k}}(N)} \tilde\lambda_F(p^n) =\begin{cases}
\sum_{i=0}^{\frac n2} p^{-\frac n2-2i} +O\left(\frac {p^c}{(k_1-1)(k_2-2)}\right)+ O\left(\frac {p^{-d}}{(k_1-k_2+1)(k_1+k_2-3)}\right), &\text{if $n$ is even}\\
O\left(\frac {p^c}{(k_1-1)(k_2-2)}\right)+ O\left(\frac {p^{-d}}{(k_1-k_2+1)(k_1+k_2-3)}\right), &\text{if $n$ is odd}\end{cases},
$$
for some constants $c,d>0$.
\end{enumerate}
\end{thm}

Since $\tilde\lambda_F$ is multiplicative, we have proved Theorem \ref{sum-formula}.

For one-level density, the root number $\epsilon(\pi_F)$ did not play a role. However, higher level density depends on the root number.

Let $S^{\pm}_{\underline{k}}(N)$ be the subspace of $S_{\underline{k}}(N)$ with the root number $\epsilon(\pi_F)=\pm 1$.
Let $HE^{\pm}_{\underline{k}}(N)$ be a basis of $S^\pm_{\underline{k}}(N)$ consisting of Hecke eigenforms outside $N$, and
denote $|HE^{\pm}_{\underline{k}}(N)|=d^\pm_{\underline{k},N}$.

When $N=1$ (i.e., level one case), we have $\epsilon(\pi_F)=(-1)^{k_2}$ (\cite{Sch}). (In this case, $k_1-k_2$ should be even.)
Hence $HE^+_{\underline{k}}(1)=HE_{\underline{k}}(1)$ when $k_2$ is even, and
$HE^-_{\underline{k}}(1)=HE_{\underline{k}}(1)$ when $k_2$ is odd.

However, when $N$ is large, we expect

\begin{Conj}\label{conj}
$d^\pm_{\underline{k},N}=\frac 12 d_{\underline{k},N}+O(N^{9-\epsilon})$,
and Theorem \ref{sum-formula} is true for each subspace ${HE}^\pm_{\underline{k}}(N)$, namely, for example, in level aspect, ($m=\prod_p p^{v_p(m)}$)
$$\frac{1}{d^+_{\underline{k},N}}\sum_{F\in HE^+_{\underline{k}}(N)} \tilde\lambda_F(m) =\delta_\square m^{-\frac 12}\prod_{p|m} (1+p^{-2}+\cdots+p^{-v_p(m)}) +O(N^{-2}m^{c}),
$$
\end{Conj}

We assume Conjecture \ref{conj} in this paper.

The calculation of the $n$-level density is well-known.
But for the sake of completeness, we give an outline. We follow closely \cite{CK2}, \cite{Rub}.

\subsection{The case $\epsilon( \pi_F)= 1$}

We denote the non-trivial zeros of $L(s, \pi_F, Spin)$ by $\sigma_{F,i}=\frac{1}{2}+\sqrt{-1} \gamma_{F,i}$. Without assuming the GRH for $L(s,\pi_F, \text{Spin})$, we can order them as
\begin{equation*}
\cdots \leq {\rm Re}({\gamma_{F,-2}}) \leq {\rm Re}({\gamma_{F,-1}}) \leq 0 \leq {\rm Re}({\gamma_{F,1}}) \leq {\rm Re}({\gamma_{F,2}}) \leq \cdots.
\end{equation*}

Let $\phi(x_1,...,x_n)=\phi_1(x_1)\cdots \phi_n(x_n)$, where each $\phi_i$ is an even Schwartz function and $\hat \phi(u_1,...,u_n)=\hat \phi_1(u_1)\cdots \hat\phi_n(u_n)$.
For a fixed $n>0$, assume the Fourier transform $\hat{\phi_i}$ of $\phi_i$ is supported in $(-\beta_n,\beta_n)$ for $i=1,\dots,n$.
($\beta_n<1$ can be explicitly determined.)
The $n$-level density function is
\begin{equation}\label{n-level}
D^{(n)}(\pi_F, \phi,{\rm Spin})
={\sum}_{j_1,\cdots,j_n}^*\phi\left(\gamma_{j_1}\frac{\log c_{\underline{k},N}}{2 \pi},\gamma_{j_2}\frac{\log c_{\underline{k},N}}{2 \pi},\dots,\gamma_{j_n}\frac{\log c_{\underline{k},N}}{2 \pi}\right)
\end{equation}
where $\sum_{j_1,...,j_n}^*$ is over $j_i=\pm 1,\pm 2,...$ with $j_{a}\ne \pm j_{b}$ for $a\ne b$, and
$\log c_{\underline{k},N}=\frac 1{d^+_{\underline{k},N}} \sum_{F\in HE^+_{\underline{k}}(N)} \log c(F)$ for $\underline{k}=(k_1,k_2)$, and $c(F)=(k_1+k_2)^2(k_1-k_2+1)^2 q(F)$ is the analytic conductor. Let
$$-\frac {L'}L(s,\pi_F,\text{Spin})=\sum_{n=1}^\infty \Lambda(n)a_F(n) n^{-s}.
$$
Recall the one-level density function (\cite{KWY}, page 68):

\begin{eqnarray}\label{explicit}
&& \frac 1{d_{\underline{k},N}} \sum_{F\in HE_{\underline{k}}(N)} D^{(1)}(\pi_F, \phi,{\rm Spin})=\widehat{\phi}(0)-
\frac {2}{d_{\underline{k},N}\log c_{\underline{k},N}} \sum_{F\in HE_{\underline{k}}(N)}\sum_p \frac{a_{F}(p)}{\sqrt{p}}\widehat{\phi}\left( \frac{\log p}{\log c_{\underline{k},N}}\right)\\
&&\phantom{xxxxxxxxxxxxxx} -\frac{2}{d_{\underline{k},N}\log c_{\underline{k},N}} \sum_{F\in HE_{\underline{k}}(N)}\sum_p \frac{a_{F}(p^2)}{p}\widehat{\phi}\left( \frac{2\log p}{\log c_{\underline{k},N}}\right)+O\left( \frac{1}{\log c_{\underline{k},N}} \right). \nonumber
\end{eqnarray}

Let $\underline{L}$ run over all ways of decomposing $\{1,...,n\}$ into disjoint subsets $[L_1,...,L_{\nu}]$. Let $\nu=\nu(\underline{L})$. For each $l=1,...,\nu$, let $\Phi_l=\prod_{i\in S_l} \phi_i$.

By Rubinstein \cite{Rub},
\begin{eqnarray} \label{n-level}
&&\frac{1}{d^+_{\underline{k},N}}\sum_{F \in HE^+_{\underline{k},N}} D^{(n)}(\pi_F, \phi,{\rm Spin}) = \frac{1}{d^+_{\underline{k},N}} \sum_{F \in HE^+_{\underline{k},N}}
\sum_{\underline{L}}(-2)^{n-\nu(\underline{L})}\\
                             &\cdot& \prod_{l=1}^{\nu(\underline{L})}\left( |L_l|-1 \right)! \left( \int_{\R} \Phi_l(x)dx -\frac{2}{\log R} \Sigma_{l,1}(\Phi_l) - \frac{2}{\log c_{\underline{k},N}} \Sigma_{l,2}(\Phi_l) + O\left( \frac{1}{\log c_{\underline{k},N}} \right) \right),  \nonumber
\end{eqnarray}
where
\begin{eqnarray*}
\Sigma_{l,1}(\Phi_l) &=& \sum_{p}\frac{a_{F}(p)\log p}{\sqrt{p}}\widehat{\Phi}_l \left( \frac{\log p}{\log c_{\underline{k},N}}\right),\\
\Sigma_{l,2}(\Phi_l) &=& \sum_{p}\frac{a_{F}(p^2)\log p}{p}\widehat{\Phi}_l \left( \frac{2 \log p}{\log c_{\underline{k},N}}\right).
\end{eqnarray*}

Note that
$$\int_{\R} \Phi_l(x)dx -\frac{2}{\log c_{\underline{k},N}} \Sigma_{l,1}(\Phi_l) - \frac{2}{\log c_{\underline{k},N}} \Sigma_{l,2}(\Phi_l)= D^{(1)}(\pi_F, \Phi_l)  + O\left( \frac{1}{\log c_{\underline{k},N}}\right).
$$
We show, following Lemma 2 in \cite{Rub}, that we can ignore $O\left( \frac{1}{\log c_{\underline{k},N}} \right)$ terms in
$(\ref{n-level})$.

\begin{lem}
\begin{eqnarray*} \label{removing-tail}
&& \frac{1}{d^+_{\underline{k},N}}\sum_{F \in HE^+_{\underline{k},N}} \prod_{l=1}^{\nu(\underline{L})}\left(\int_{\R} \Phi_l(x)dx -\frac{2}{\log c_{\underline{k},N}} \Sigma_{l,1}(\Phi_l) - \frac{2}{\log c_{\underline{k},N}} \Sigma_{l,2}(\Phi_l) \right)
\\
&=& \frac{1}{d^+_{\underline{k},N}}\sum_{F \in HE^+_{\underline{k},N}} \prod_{l=1}^{\nu(\underline{L})}\left(D^{(1)}(\pi_F, \Phi_l ) \right)+  O\left( \frac{1}{\log c_{\underline{k},N}} \right). \nonumber
\end{eqnarray*}
\end{lem}

We prove two lemmas analogous to Claim 2 and Claim 3 in Rubinstein \cite{Rub}.

\begin{lem} (Claim 2 of \cite{Rub}) \label{lemma1}
\begin{equation*}\label{non-square}
\sum_{F \in HE^+_{\underline{k},N}}
\sum_{m_i\geq 1\atop m_1m_2\cdots m_a \neq \square} \prod_{j=1}^a \frac{\Lambda(m_j)a_{F}(m_j)}{\sqrt{m_j}}\widehat{\Phi}_{l_j}\left( \frac{\log m_j}{\log c_{\underline{k},N}}\right) \ll d^+_{\underline{k},N}(\log c_{\underline{k},N})^{a-1}.
\end{equation*}
where $m_i$'s are a prime or a square of a prime.
\end{lem}
\begin{proof} By changing the order of the sums, we need to consider, for $p_1,p_2,\cdots,p_r,q_1,q_2,\cdots,q_t$, distinct primes,

\begin{eqnarray}\label{sum}
\sum_{F \in HE^+_{\underline{k},N}}
a_{F}(p_1)^{e_1}\cdots a_{F}(p_r)^{e_r} a_{F}(q_1^2)^{\gamma_1}\cdots a_{F}(q_t^2)^{\gamma_t},
\end{eqnarray}
where $e_1+ \cdots +e_r + \gamma_1 + \cdots +\gamma_t=a$.

Now, by rearranging, if there exists $e_i=1$ for some $i$, we can assume that $e_1=\cdots=e_b=1$ and $e_{b+1}>1,...e_r>1$.
Then use the fact that $\lambda_F(p)^e=\sum_{i=0}^e c_i \lambda_F(p^i)$ for some constants $c_i$. Hence
$a_{F}(p_1)^{e_1}\cdots a_{F}(p_r)^{e_r} a_{F}(q_1^2)^{\gamma_1}\cdots a_{F}(q_t^2)^{\gamma_t}$ is the sum of
$\lambda_F(p_1\cdots p_b n)$, where $n$ is not divisible by $p_1,...,p_b$. Hence by Theorem \ref{sum-formula}, $(\ref{sum})$ gives rise to an error term.
So we have the result.

Suppose $e_i>1$ for all $i$. Then $e_i\geq 3$ for some $i$ by assumption. Then we have the result
by using the bounds (\cite[(4.10) and (4.11)]{CK2}) :
\begin{eqnarray}
&& \sum_{p \leq R^{\beta_n}} \frac{|a_F(p)|^i(\log p)^i}{\sqrt{p^i}}\ll \begin{cases} R^{\frac {\beta_n}2}, &\text{if $i=1$}\\
(\log R)^2, &\text{if $i=2$}\\
O(1), &\text{if $i\geq 3$}.
\end{cases}, \label{bounds1}\\
&& \sum_{p \leq R^{\beta_n/2}} \frac{|a_F(p^2)|^{j}(\log p)^{j}}{p^{j}}
\ll \begin{cases} O(\log R), &\text{if $j=1$}\\
 O(1), &\text{if $j \geq 2$}.
\end{cases} \label{bounds2}
\end{eqnarray}
\end{proof}

\begin{lem} (Claim 3 of \cite{Rub}) \label{lemma2}
\begin{eqnarray*}\label{square}
&&\hspace{0.4in} \frac{1}{d^+_{\underline{k},N}} \left(\frac {-2}{\log c_{\underline{k},N}}\right)^a \sum_{F \in HE^+_{\underline{k},N}}
\sum_{m_i\geq 1\atop m_1m_2\cdots m_a = \square} \left( \prod_{j=1}^l \frac{\Lambda(m_j)a_{F}(m_j)}{\sqrt{m_j}}\widehat{\Phi}_{l_j}\left( \frac{\log m_j}{\log c_{\underline{k},N}}\right) \right)\\
&=& \sum_{S_2 \subseteq S\atop |S_2| \, {\rm even}} \left( \left( \frac{\delta_{\pi}}{2}\right)^{|S_2^c|} \prod_{l \in S_2^c} \int_{\mathbb{R}} \widehat{\Phi}_l (u)du \right)
\left( \sum_{P_2}2^{|S_2|/2} \prod_{j=1}^{|S_2|/2} \int_{\mathbb{R}}|u|\widehat{\Phi}_{a_j}(u)\widehat{\Phi}_{b_j}(u)du\right) +O\left(\frac {1}{\log c_{\underline{k},N}}\right),\nonumber
\end{eqnarray*}
where $m_i$'s are a prime or a square of a prime, and $S=\{l_1,...,l_a\}$, $\sum_{S_2 \subseteq S\atop |S_2| \, {\rm even}}$ is over all subsets $S_2$ of $S$ whose size is even, and $\sum_{P_2}$ is over all ways of pairing up the elements of $S_2$.
\end{lem}
\begin{proof} First of all, in (\ref{sum}), if $e_i\geq 4$ or $\gamma_j\geq 2$ for some $i,j$, then by (\ref{bounds1}) and (\ref{bounds2}),
those terms are majorized by $d^+_{\underline{k},N}(\log c_{\underline{k},N})^{a-1}$.
Hence we only need to consider the sum
\begin{eqnarray*}
\sum_{F \in HE^+_{\underline{k},N}} a_{F}(p_1)^2\cdots a_{F}(p_r)^2 a_{F}(q_1^2)\cdots a_{F}(q_t^2),
\end{eqnarray*}
where $2r+t=a$. In this case, $S_2=\{l_1,l_2,...,l_{2r-1},l_{2r}\}$, and $S=\{l_1,...,l_{2r},l_{2r+1},...,l_a\}$.
We use the fact that
$a_F(p)^2=1+p^{-1}\times (\text{polynomials in $p^{-1}$)+sum of Hecke operators}$, and $a_F(p^2)= -1+p^{-1}\times(\text{polynomials in $p^{-1}$)+sum of Hecke operators}$.
\end{proof}

We define
$$\Delta_{l,2}(\Phi_l)=\Sigma_{l,2}(\Phi_l)+\frac 14 \Phi_l(0)\log c_{\underline{k},N}.
$$
Since $\int_{\R} \Phi_l(x)dx=\widehat{\Phi}_l(0)$,
\begin{eqnarray*}
&&(\ref{n-level})=\frac{1}{d^+_{\underline{k},N}}\sum_{F \in HE^+_{\underline{k},N}}
 \sum_{\underline{L}} (-2)^{n-\nu(\underline{L})}\prod_{l=1}^{\nu(\underline{L})}\left( |F_l|-1 \right)! \\
&&\phantom{xxxxxxxxxxxx} \cdot \left( \left(\widehat{\Phi}_l(0)+\frac 12\Phi_l(0)\right)-\frac 2{\log c_{\underline{k},N}}\Sigma_{l,1}(\Phi_l)-\frac 2{\log c_{\underline{k},N}}\Delta_{l,2}(\Phi_l)\right)+O\left(\frac 1{\log c_{\underline{k},N}}\right).
\end{eqnarray*}

The following three lemmas enable us to find an explicit expression for $(\ref{n-level})$.
\begin{lem}
\begin{eqnarray*} \label{new-sum}
\frac{1}{d^+_{\underline{k},N}}\left(\frac {-2}{\log c_{\underline{k},N}}\right)^a \sum_{F \in HE^+_{\underline{k},N}}
 \Delta_{l_1,2}(\Phi_{l_1})\cdots \Delta_{l_a,2}(\Phi_{l_a})\ll \frac {1}{\log c_{\underline{k},N}}.
\end{eqnarray*}
\end{lem}

\begin{lem}
\begin{equation*}\label{new-sum-2}
\hspace{0.4in}\frac{1}{d^+_{\underline{k},N}} \left( \frac{-2}{\log c_{\underline{k},N}}\right)^a \sum_{F \in HE^+_{\underline{k},N}} \Sigma_{l_1,1}(\Phi_{l_1})\cdots \Sigma_{l_r,1}(\Phi_{l_r}) \Delta_{l_{r+1},2}(\Phi_{l_{r+1}})\cdots \Delta_{l_a,2}(\Phi_{l_a})=O\left(\frac {1}{\log c_{\underline{k},N}}\right).
\end{equation*}
\end{lem}

\begin{lem}
\begin{eqnarray*}
 &&\frac{1}{d^+_{\underline{k},N}} \left( \frac{-2}{\log c_{\underline{k},N}}\right)^a \sum_{F \in HE^+_{\underline{k},N}} \Sigma_{l_1,1}(\Phi_{l_1})\cdots \Sigma_{l_a,1}(\Phi_{l_a}) \\
 &=&\begin{cases} \displaystyle{\sum_{P_2} 2^{\frac a2} \prod_{i=1}^a \int_{\Bbb R} |u| \widehat{\Phi}_{l_{a_i}}(u)\widehat{\Phi}_{l_{b_i}}(u)\, du} + O\left(\frac {1}{\log c_{\underline{k},N}}\right), &\text{if $a$ is {\rm even}} \\
 \displaystyle{O\left(\frac {1}{\log c_{\underline{k},N}}\right)}, &\text{{\rm otherwise.}}
\end{cases}
\nonumber \end{eqnarray*}
Here $\sum_{P_2}$ is over all ways of paring up the elements of $\{1,2,\dots,a\}$. If $a=2r$, it runs through products of 2-cycles of the form $(l_1 l_2)\cdots (l_{2r-1} l_{2r})$. There are $\displaystyle\frac {(2r)!}{2^r r!}$ of them.
\end{lem}

Therefore, we have proved

\begin{thm}\label{+1}
\begin{eqnarray} \label{n-level+1}
&& \frac 1{d^+_{\underline{k},N}} \sum_{F \in HE^+_{\underline{k},N}} D^{(n)}(\pi_F,\phi,{\rm Spin})= \sum_{\underline{L}} (-2)^{n-\nu(\underline{L})}\prod_{l=1}^{\nu(\underline{L})}\left( |F_l|-1 \right)! \\
&& \sum_{S\, {\rm even}} \left( \prod_{l\in S^c} \left(\widehat{\Phi_l}(0)+\frac 12\Phi_l(0)\right) \right)
\left(\sum_{P_2} 2^{\frac {|S|}2} \prod_{i=1}^{\frac {|S|}2} \int_{\Bbb R} |u| \widehat{\Phi}_{l_{a_i}}(u)\widehat{\Phi}_{l_{b_i}}(u)\, du\right) +O\left(\frac 1{\log c_{\underline{k},N}}\right),\nonumber
\end{eqnarray}
where $S$ runs through the subset of even cardinality in $\{1,...,\nu(\underline{L})\}$, and $\sum_{P_2}$ is over all ways of pairing up the elements of $S$.
\end{thm}

We summarize it as
\begin{thm}\label{n-level-spin-even}
Let $\phi(x_1,...,x_n)=\phi_1(x_1)\cdots \phi_n(x_n)$, where each $\phi_i$ is an even Schwartz function and $\hat \phi(u_1,...,u_n)=\hat \phi_1(u_1)\cdots \hat\phi_n(u_n)$. Assume the Fourier transform $\hat{\phi_i}$ of $\phi_i$ is supported in $(-\beta_n,\beta_n)$ for $i=1,\cdots,n$.
Then
$$
\frac 1{d^+_{\underline{k},N}} \sum_{F\in HE^+_{\underline{k},N}} D^{(n)}(\pi_F,\phi,{\rm Spin})=
\int_{\Bbb R^n} \phi(x)W({\rm SO(even)})(x)\, dx+ O\left(\frac {\omega(N)}{\log c_{\underline{k},N}}\right).
$$
\end{thm}

\subsection{The case $\epsilon( \pi_F)=-1$ }

If $\epsilon(\pi_F)=-1$, $L(s, \pi_F,\text{Spin})$ always has a family zero at $s=\frac 12$. By Rubinstein \cite{Rub},
the $n$-level density function is

\begin{eqnarray*} \label{ep-1}
&& D^{(n)}(\pi_F,\phi,{\rm Spin}) ={\sum}_{j_1,\cdots,j_n}^*\phi\left(\gamma_{j_1}\frac{\log c_{\underline{k},N}}{2 \pi},\gamma_{j_2}\frac{\log c_{\underline{k},N}}{2 \pi},\dots,\gamma_{j_n}\frac{\log c_{\underline{k},N}}{2 \pi}\right)\\
&=& {\sum}_{j_1 \neq 0,\dots,j_n \neq 0}^*\phi\left(\gamma_{j_1}\frac{\log c_{\underline{k},N}}{2 \pi},\gamma_{j_2}\frac{\log c_{\underline{k},N}}{2 \pi},\dots,\gamma_{j_n}\frac{\log c_{\underline{k},N}}{2 \pi}\right) \nonumber\\
&+& \sum_{\nu=1}^n {\sum}_{j_{\nu}=0,j_k \neq 0, k \neq \nu}^*\phi\left(\gamma_{j_1}\frac{\log c_{\underline{k},N}}{2 \pi},\gamma_{j_2}\frac{\log c_{\underline{k},N}}{2 \pi},\dots,\gamma_{j_{\nu-1}}\frac{\log c_{\underline{k},N}}{2 \pi},0,\gamma_{j_{\nu+1}}\frac{\log c_{\underline{k},N}}{2 \pi},\dots,\gamma_{j_n}\frac{\log c_{\underline{k},N}}{2 \pi}\right).\nonumber
\end{eqnarray*}

By Rubinstein \cite{Rub}, the first term gives rise to
\begin{eqnarray*} \label{-1}
&& \hspace{1cm}  \frac 1{d^-_{\underline{k},N}} \sum_{F \in HE^-_{\underline{k},N}} {\sum}_{j_1\ne 0,\dots,j_n\ne 0}^*
\phi\left(\gamma_{j_1}\frac{\log c_{\underline{k},N}}{2 \pi},\gamma_{j_2}\frac{\log c_{\underline{k},N}}{2 \pi},\dots,\gamma_{j_n}\frac{\log c_{\underline{k},N}}{2 \pi}\right) \\
&=& \sum_{\underline{L}}(-2)^{n-\nu(\underline{L})}\prod_{l=1}^{\nu(\underline{L})}\left( |F_l|-1 \right)! \left( \widehat{\Phi_l}(0) -\frac{2 \Sigma_{l,1}(\Phi_l)}{\log c_{\underline{k},N}} - \frac{2 \Sigma_{l,2}(\Phi_l)}{\log c_{\underline{k},N}} -\Phi_l(0) + O\left( \frac{1}{\log c_{\underline{k},N}}\right)\right) \nonumber\\
&=& \sum_{\underline{L}}(-2)^{n-\nu(\underline{L})}\prod_{l=1}^{\nu(\underline{L})}\left( |F_l|-1 \right)! \left( \left(\widehat{\Phi_l}(0)-\frac 12\Phi_l(0)\right)- \frac {2 \Sigma_{l,1}(\Phi_l)}{\log c_{\underline{k},N}} -\frac {2 \Delta_{l,2}(\Phi_l)}{\log c_{\underline{k},N}} \right)+O\left(\frac 1{\log c_{\underline{k},N}}\right) .\nonumber
\end{eqnarray*}

It is equal to
\begin{eqnarray*}
&& \sum_{\underline{L}}(-2)^{n-\nu(\underline{L})}\prod_{l=1}^{\nu(\underline{L})}\left( |F_l|-1 \right)! \sum_{S\, \text{even}} \left( \prod_{l\in S^c} \left(\widehat{\Phi_l}(0-\frac {1}{2}\Phi_l(0)\right) \right)\\
&& \phantom{xxxxxxxxxxxxxxxxxx} \cdot \left(\sum_{P_2} 2^{\frac {|S|}2} \prod_{i=1}^{\frac {|S|}2} \int_{\Bbb R} |u| \widehat{\Phi_{l_{a_i}}}(u)\widehat{\Phi_{l_{b_i}}}(u)\, du\right) +O\left(\frac 1{\log c_{\underline{k},N}}\right),
\end{eqnarray*}
which equals the $n$-level density of the symplectic type.

We summarize it as
\begin{thm}\label{n-level-spin-odd}
Let the notations be as in Theorem \ref{n-level-spin-even}.
Then
$$
\frac 1{d^-_{\underline{k},N}} \sum_{F\in HE^-_{\underline{k},N}} D^{(n)}(\pi_F,\phi,{\rm Spin})=
\int_{\Bbb R^n} \phi(x)W({\rm SO(odd)})(x)\, dx+ O\left(\frac {\omega(N)}{\log c_{\underline{k},N}}\right).
$$
\end{thm}

\section{$n$-level density of degree 5 standard $L$-functions}\label{Standard}

Let $L(s,\pi_F, \text{St})$ be the degree 5 standard $L$-function. Let
$$L(s,\pi_F, \text{St})=\sum_{n=1}^\infty \mu_F(n) n^{-s}.
$$
It satisfies the functional equation: Let $\Lambda(s,\pi_F, \text{St})=q(F,\text{St})^{\frac s2} \Gamma_\Bbb R(s) \Gamma_\Bbb C(s+k_1-1)\Gamma_\Bbb C(s+k_2-2)L(s,\pi_F,\text{St})$. Then
$$\Lambda(s,\pi_F,\text{St})=\epsilon(\pi_F,\text{St})\Lambda(1-s,\pi_F,\text{St}),
$$
where $\epsilon(\pi_F,\text{St})\in\{\pm 1\}$ and $N\le q(F,{\rm St})\le N^{28}$. Lapid \cite{La} showed

\begin{prop} Let $\pi_F$ be as above. Then $\epsilon(\pi_F,{\rm St})=1$.
\end{prop}

In fact, Lapid proved it only for globally generic cusp forms. However, holomorphic cusp forms are always in the same $L$-packet with a globally generic cusp form. Hence the result follows.
Because of the above proposition, we expect that the symmetry type of $L(s,\pi_F,\text{St})$ is ${\rm Sp}$.
However, we need the following conjecture. We showed it in \cite[Proposition 9.5]{KWY} when $m$ is of the form
$m=p_1^{a_1}\cdots p_r^{a_r}$, where $p_i$'s are distinct primes, and $a_i=1,2$ for each $i$.

\begin{Conj}\label{stan}
Put $\underline{k}=(k_1,k_2),\ k_1\ge k_2\ge 3$.
\begin{enumerate}
\item (level-aspect) Fix $k_1,k_2$. Let $m=\prod_{p|m} p^{v_p(m)}$. Then as $N\to\infty$,
$$\frac{1}{d_{\underline{k},N}}\sum_{F\in HE_{\underline{k}}(N)} \mu_F(m) =\prod_{p|m} (\delta_{\square}+ p^{-1}h(p^{-1}))+ O(N^{-2}m^c),
$$
where $\delta_{\square}=\begin{cases} 1, &\text{if $m$ is a square}\\0, &\text{otherwise}\end{cases}$,
\newline and $h$ is a polynomial with integer coefficients and $c>0$ is a constant.
\item (weight-aspect) Fix $N$. Then as $k_1+k_2\to\infty$,
$$\frac{1}{d_{\underline{k},N}}\sum_{F\in HE_{\underline{k}}(N)} \mu_F(m) =\prod_{p|m} (\delta_\square + p^{-1}h(p^{-1}))+
O\left(\frac {m^c}{(k_1-1)(k_2-2)}\right)+ O\left(\frac {m^d}{(k_1-k_2+1)(k_1+k_2-3)}\right).
$$
\end{enumerate}
\end{Conj}

Let
$$-\frac {L'}L(s,\pi_F,{\rm St})=\sum_{n=1}^\infty \Lambda(n)b_F(n) n^{-s}.
$$

We use the fact that
$b_F(p)^2=1+p^{-1}\times(\text{polynomials in $p^{-1}$)+sum of eigenvalues of Hecke operators}$, and $b_F(p^2)=1+p^{-1}\times(\text{polynomials in $p^{-1}$)+sum of eigenvalues of Hecke operators}$.

Let $\phi$ be a Schwartz function which is even and whose Fourier transform has a compact support.
Define
\begin{equation*}
D(\pi_F,\phi,\text{St}) = \sum_{\gamma_{F}}\phi\left( \frac{\gamma_{F}}{2\pi} \log c_{\underline{k},st,N}\right),
\end{equation*}
where $\log c_{\underline{k},st,N}=\frac 1{d_{\underline{k},N}} \sum_{F\in HE_{\underline{k}}(N)} \log c(F,\text{St})$, and
$c(F,\text{St})=(k_1 k_2)^2 q(F,St)$ is the analytic conductor.

We showed in \cite{KWY} that

$$\frac 1{d_{\underline{k},N}} \sum_{F\in HE_{\underline{k}}(N)}  D(\pi_F,\text{St},\phi)=\widehat{\phi}(0)-\frac 12 \phi(0)
+O\left( \frac{1}{\log c_{\underline{k},st,N}} \right)
=\int_\Bbb R \phi(x)W(\text{Sp})(x)\, dx+ +O\left( \frac{\omega(N)}{\log c_{\underline{k},st,N}} \right).
$$

Let $\phi(x_1,...,x_n)=\phi_1(x_1)\cdots \phi_n(x_n)$, where each $\phi_i$ is an even Schwartz function and $\hat \phi(u_1,...,u_n)=\hat \phi_1(u_1)\cdots \hat\phi_n(u_n)$.
For a fixed $n>0$, assume the Fourier transform $\hat{\phi_i}$ of $\phi_i$ is supported in $(-\beta_n,\beta_n)$ for $i=1,\dots,n$.
($\beta_n < 1$ can be explicitly determined.)
The $n$-level density function is
\begin{equation*}\label{n-level-st}
D^{(n)}(\pi_F, \phi,\text{St})
={\sum}_{j_1,\cdots,j_n}^*\phi\left(\gamma_{j_1}\frac{\log c_{\underline{k},st,N}}{2 \pi},\gamma_{j_2}\frac{\log c_{\underline{k},st,N}}{2 \pi},\dots,\gamma_{j_n}\frac{\log c_{\underline{k},st,N}}{2 \pi}\right)
\end{equation*}
where $\sum_{j_1,...,j_n}^*$ is over $j_i=\pm 1,\pm 2,...$ with $j_{a}\ne \pm j_{b}$ for $a\ne b$.
Then as in the degree 4 spinor $L$-functions, we can show Theorem \ref{n-level-sp}.
(We can prove the analogues of (\ref{bounds1}) and (\ref{bounds2}) for $b_F(p^i)$ from \cite[Appendix]{KWY}:
Let $\Pi$ be the cuspidal representation of $GL_5/\Q$ such that $L(s,\Pi)=L(s,\pi_F,{\rm St})$. Then
(\ref{bounds1}) follows from the fact that $L(s, \Pi\times\Pi)$ converges absolutely for $Re(s)>1$ and has a simple pole at $s=1$;
(\ref{bounds2}) follows from the fact that $|b_F(p^2)|\ll |b_F(p)|^2$ for any $p$, and the Ramanujan bound
$|b_F(p^2)|\leq 5 p^{1-\frac 2{26}}$.)

\section{Paramodular forms}

\def\para{{\mathrm{para}}}

In this section, we fix a square free positive integer $N$.
We deal with a compact subgroup $K_p^\para(N)$ of level $N$ in $G(\Q_p)$, which is defined by
\[
K_p^\para(N)=xM_4(\Z_p)x^{-1}\cap G(\Q_p),\quad x=\diag(1,1,N,1),
\]
where $G$ means $GSp_4$ as in \cite[Section 2]{KWY}.
We have an open compact subgroup $K^\para(N)=\prod_p K_p^\para(N)$ in $G(\A_\fin)$.
Furthermore, an arithmetic subgroup $\Gamma^\para(N)$ of $G(\Q)$ is given by $\Gamma^\para(N)=G(\Q)\cap (G(\R)K^\para(N))$, called the paramodular subgroup.
An element $w_{N,p}$ in $G(\Q_p)$ is given by
\[
w_{N,p}=\begin{pmatrix}0&0&0&1 \\ 0&0&-1&0 \\ 0&-N&0&0 \\ N&0&0&0  \end{pmatrix} \in  G(\Q_p).
\]
Then, the Atkin-Lehner involution on $G(\Q_p)$ is provided by the double coset
\[
K^\para_p(N)w_{N,p}K^\para_p(N)=w_{N,p}K^\para_p(N)=K^\para_p(N)w_{N,p}.
\]
Let $\underline{k}=(k_1,k_2)$, $k_1\geq k_2\geq 4$, and $S_{\underline{k}}(\Gamma^\para(N))$ denote the space of paramodular forms of weight $\underline{k}$ with the trivial central character.
Set $d_{\underline{k}}^\para=\dim S_{\underline{k}}(\Gamma^\para(N))$.
Then by \cite{IK},
\begin{eqnarray*}
&& d_{\underline{k}}^\para=2^{-7}3^{-3}5^{-1} (k_1-1)(k_2-2)(k_1-k_2+1)(k_1+k_2-3) \prod_{p|N} (p^2+1)\\
&& \phantom{xxxxxxxxxxxxxxxxx} + O(N(k_1-k_2+1)(k_1+k_2-3))+O(N(k_1-1)(k_2-2)).
\end{eqnarray*}
Here note that $\prod_{p|N} (p^2+1)=c_N N^2$ for some constant $1<c_N<5$.

We obtain the Hecke operator $T_m'$ as in \cite[Section 8]{KWY}.
Define, for $M|N$,
the Atkin-Lehner involution $w_{N,M}$ on $S_{\underline{k}}(\Gamma^\para(N))$, where $w_{N,M}$ is induced from the coset $\prod_{p|M}K^\para_p(N)w_{N,p}$ in $\prod_{p|M}G(\Q_p)$.

\begin{thm}\label{estimate}
Suppose $k_1\geq k_2\geq 4$.
There exist absolute constants $a$ and $b$ such that for each prime $p\nmid N$, square-free natural number $M$ dividing $N$, and $m=\diag(p^{a_1},p^{a_2},p^{-a_1+\kappa},p^{-a_2+\kappa})$, $a_1,a_2,\kappa\in \Z$ satisfying $0\le a_2\le a_1\le \kappa$, we have
\[
\frac{1}{d_{\underline{k},N}^\para}{\rm tr}(T'_m w_{N,M}|S_{\underline{k}}(\Gamma^\para(N))) = B_1+B_2+ O(\frac{p^{a\kappa+b}}{(k_1-k_2+1)(k_1-1)(k_2-2)})\ \  (k_1+k_2\to \infty),
\]
\[
B_1=O(\frac{p^{-\frac{\kappa}{2}}}{(k_1-1)(k_2-2)}),\ B_2=O(\frac{p^{-\frac{\kappa}{2}}}{(k_1-k_2+1)(k_1+k_2-3)})
\]
if $m\not\in Z_G(\Q)$ or $M\neq 1$.
\end{thm}
\begin{proof}
We recall the notations as in Sections 5 and 6 of \cite{KWY}.
For the trace formula, we choose a test function $f=f_\xi h$ as
\[
h=f_{S',\alpha} \, \tilde f \, \Big( \bigotimes_{p\in S\setminus S'\sqcup \{v\mid N\inf\} }  \mathrm{char}_{K_p} \Big)
\]
where $\tilde f$ denotes $\big(\otimes_{p|N,\, p\nmid M}\mathrm{char}_{K_p^\para(N)} \big) \bigotimes \big(\otimes_{p| M}\mathrm{char}_{w_{N,p}K_p^\para(N)} \big)$.
By the same argument as the proof of \cite[Proposition 6.3]{KWY}, one can show that there exist positive constants $a'$ and $b'$ such that
\[
I_2(f)\times \vol(K^\para(N))^{-1}\times |\nu(\alpha)|^{-\frac{k_1+k_2}{2}}_{S'} =O(p_{S'}^\kappa(k_1-k_2+1)(k_1+k_2-3))
\]
\[
I_3(f)\times \vol(K^\para(N))^{-1}\times |\nu(\alpha)|^{-\frac{k_1+k_2}{2}}_{S'} =O(p_{S'}^\kappa (k_1-1)(k_2-2)),
\]
\[
\{I_4(f)+I_6(f)\}\times \vol(K^\para(N))^{-1}\times |\nu(\alpha)|^{-\frac{k_1+k_2}{2}}_{S'} =O(p_{S'}^{a'\kappa +b'}(k_1+k_2-3)),
\]
\[
I_5(f)\times \vol(K^\para(N))^{-1}\times |\nu(\alpha)|^{-\frac{k_1+k_2}{2}}_{S'} =O(p_{S'}^{a'\kappa +b'}(k_1+k_2-3))
\]
for any $(k_1,k_2)$, $\kappa\geq 1$, $S'$, and $f_{S',\alpha}$, which satisfy the conditions $k_1\geq k_2\geq 3$, $f_{S',\alpha}\in H^\mathrm{ur}(G(\Q_{S'}))^\kappa$, and $N$ is prime to $\prod_{p\in S'}p$.
Notice that an important result \cite[Proposition 8.7]{ST} was used for the proof.
Thus, this theorem is derived from the above estimates.
\end{proof}

Let $S^{\text{non-C}}_{\underline{k}}(\Gamma^\para(N))$ denote the space of paramodular forms whose associated representations are not related to the Saito-Kurokawa representations.
We write $S^{\mathrm{CAP}}_{\underline{k}}(\Gamma^\para(N))$ for the subspace consisting of Saito-Kurokawa liftings in $S_{\underline{k}}(\Gamma^\para(N))$ (i.e., it is called the Maass space).
Hence, $S^{\mathrm{CAP}}_{\underline{k}}(\Gamma^\para(N))$ is the orthogonal complement of $S^{\text{non-C}}_{\underline{k}}(\Gamma^\para(N))$ in $S_{\underline{k}}(\Gamma^\para(N))$ by the Petersson inner product.
Namely, we have an isomorphism
\[
S^{\text{non-C}}_{\underline{k}}(\Gamma^\para(N))\cong\bigoplus_{\pi=\pi_\inf\otimes \pi_\fin} N_\pi(K^\para(N))
\]
where $\pi$ moves over automorphic representations of $G(\A)$ such that $\pi$ has the trivial central character, $\pi$ is not a Saito-Kurokawa representation, $\pi_\inf$ is isomorphic to the holomorphic discrete series of $G(\R)$ with the Harish-Chandra parameter $(k_1-1,k_2-2)$, and the subspace $N_\pi(K^\para(N))$ of $K^\para(N)$-fixed vectors in $\pi_\fin$ is not trivial.
By \cite[Theorem 3.3]{Wei2} or \cite[a comment for $n=2$ after Theorem A]{KrSh}, for each above $\pi_\fin=\otimes_{v<\inf}\pi_v$, if $\pi_v$ is spherical, then $\pi_v$ satisfies the Ramanujan conjecture.
Hence, all spherical representations $\pi_v$ belong to the class I in \cite[Table A.13 in p.293]{RS} (see also \cite[Table 3]{SchIwa}).
Hence, by \cite{SchIwa,RS}, one has an isomorphism
\[
N_\pi(K^\para(N))\cong \bigotimes_{v<\inf}N_{\pi_v}(K^\para_v(N))
\]
where $N_{\pi_v}(K^\para_p(N))$ denotes the subspace of $K^\para_v(N)$-fixed vectors in $\pi_v$, and for each prime $p|N$,  $\pi_p$ satisfies (i) or (ii);
\begin{itemize}
\item[(i)] $\pi_p$ is spherical and $\dim N_{\pi_p}(K_p^\para(N))=2$ ,
\item[(ii)] $\pi_p$ is non-spherical and $\dim N_{\pi_p}(K_p^\para(N))=1$.
\end{itemize}
As for the case (i), the trace of the Atkin-Lehner involution on $N_{\pi_p}(K_p^\para(N))$ is zero.
When $\pi_p$ satisfies (ii), the eigenvalue of the Atkin-Lehner involution means the $\varepsilon$-factor of its local spinor L-factor (see \cite[Theorem 5.7.3 in p.185]{RS} or \cite[Proposition 1.3.1]{SchIwa}).

By \cite{SchSK}, it is obvious that $\frac{1}{d_{\underline{k},N}^\para}{\rm tr}(T'_m w_{N,M}|S^{\text{CAP}}_{\underline{k}}(\Gamma^\para(N)))$ is negligible.
Furthermore, let $S^{\text{non-C,new}}_{\underline{k}}(\Gamma^\para(N))$ denote the subspace of newforms in $S^{\text{non-C}}_{\underline{k}}(\Gamma^\para(N))$.
Here, newform means that its associated automorphic representation $\pi$ has no $ K^\para(M)$-fixed vectors for any natural number $M$ such that $M|N$ and $M<N$.
Namely,
\[
S^{\text{non-C,new}}_{\underline{k}}(\Gamma^\para(N)) \cong\bigoplus_{\pi=\otimes_v \pi_v} N_\pi(K^\para(N))
\]
where $\pi$ moves over all automorphic representations satisfying the same conditions as above and $\pi_p$ is of the case (ii) for each $p|N$.
Therefore, one has
\[
{\rm tr}(T'_m|S^{\text{non-C,new}}_{\underline{k}}(\Gamma^\para(N)))=\sum_{M|N}  (-2)^{\omega(M)}  {\rm tr}(T'_m|S^{\text{non-C}}_{\underline{k}}(\Gamma^\para(N/M)))
\]
where $\omega(M)$ is the number of distinct prime factors of $M$.

Let $S^{\text{non-C,new},+}_{\underline{k}}(\Gamma^\para(N))$ (resp. $S^{\text{non-C,new},-}_{\underline{k}}(\Gamma^\para(N))$) denote the subspace of newforms whose $\varepsilon$-factors are $1$ (resp. $-1$).
By the above mentioned arguments, one gets
\begin{multline*}
{\rm tr}(T'_m|S^{\mathrm{non-C,new},\pm}_{\underline{k}}(\Gamma^\para(N)))\\
=\frac{1}{2} \Big[ {\rm tr}(T'_m|S^{\text{non-C,new}}_{\underline{k}}(\Gamma^\para(N))) \pm (-1)^{k_2}{\rm tr}(T'_mw_{N,N}|S^{\text{non-C}}_{\underline{k}}(\Gamma^\para(N))) \Big],
\end{multline*}
Therefore, if we set
\[
d_{\underline{k}}^{\para,\mathrm{new},\pm}=\dim S^{\text{non-C,new},\pm}_{\underline{k}}(\Gamma^\para(N))),
\]
then by \cite{IK} and Theorem \ref{estimate},
\begin{eqnarray*}
&& d_{\underline{k}}^{\para,\mathrm{new},\pm}=2^{-8}3^{-3}5^{-1} (k_1-1)(k_2-2)(k_1-k_2+1)(k_1+k_2-3) \prod_{p|N} (p^2-1)\\
&& \phantom{xxxxxxxxxxxxxxxxx} + O(N(k_1-k_2+1)(k_1+k_2-3))+O(N(k_1-1)(k_2-2)).
\end{eqnarray*}
Furthermore, by Theorem \ref{estimate}, there exist absolute constants $a$ and $b$ such that for each prime $p\nmid N$ and $m=\diag(p^{a_1},p^{a_2},p^{-a_1+\kappa},p^{-a_2+\kappa}),\ a_1,a_2,\kappa\in \Z$ satisfying $0\le a_2\le a_1\le \kappa$ and $m\not\in Z_G(\Q)$,
\[
\frac{1}{d_{\underline{k},N}^{\para,\mathrm{new},\pm}}{\rm tr}(T'_m|S^{\mathrm{non-C,new},\pm}_{\underline{k}}(\Gamma^\para(N))) = B_1+B_2+ O(\frac{p^{a\kappa+b}}{(k_1-k_2+1)(k_1-1)(k_2-2)}).
\]
Hence we have proved Conjecture \ref{conj} for paramodular newforms.
Therefore, we have proved $n$-level density for spinor $L$-functions of paramodular newforms in weight aspect
(analogues of Theorem \ref{n-level-spin-even} and Theorem \ref{n-level-spin-odd} for paramodular forms).

In a similar way, we can show a simultaneous vertical Sato-Tate theorem for paramodular forms (analogue of Theorem \ref{Sato-Tate}).

\end{document}